\documentclass[11pt]{amsart}
\usepackage{amsmath, amssymb}
 \usepackage{amsmath,amscd}
\usepackage{amsfonts}
\usepackage{mathrsfs}
\usepackage[arrow,matrix,curve,cmtip,ps]{xy}
\usepackage{pb-diagram}
 \usepackage{pb-xy}

\usepackage{graphicx}

\usepackage{amsthm}

\usepackage{color}

\usepackage{xcolor}

\usepackage{setspace}

\newtheorem*{rep@theorem}{\rep@title}
\newcommand{\newreptheorem}[2]{
\newenvironment{rep#1}[1]{
 \def\rep@title{#2 \ref{##1}}
 \begin{rep@theorem}}
 {\end{rep@theorem}}}
\makeatother

\newreptheorem{theorem}{Theorem}
\newreptheorem{lemma}{Lemma}
\newreptheorem{proposition}{Proposition}
\newreptheorem{corollary}{Corollary}

\newtheorem{theorem}{Theorem}[section]
\newtheorem{lemma}[theorem]{Lemma}
\newtheorem{proposition}[theorem]{Proposition}
\newtheorem{corollary}[theorem]{Corollary}

\newtheorem*{theorem*}{Theorem}
\newtheorem*{proposition*}{Proposition}
\theoremstyle{remark}
\newtheorem{remark}[theorem]{Remark}
\newtheorem{construction}[theorem]{Construction}
\newtheorem{definition}[theorem]{Definition}
\newtheorem{example}[theorem]{Example}
\newtheorem{question}{Question}

\numberwithin{equation}{section}

\newcommand{\F}{\mathcal{F}}
\newcommand{\Z}{\mathbb{Z}}
\newcommand{\Q}{\mathbb{Q}}

\newcommand{\R}{\mathbb{R}}
\newcommand{\C}{\mathcal{C}}
\newcommand{\T}{\mathbb{T}}
\newcommand{\K}{\mathcal{K}}

\newcommand{\rank}{\operatorname{rank}}
\newcommand{\dual}{\operatorname{dual}}

\newcommand{\real}{\operatorname{re}}

\newcommand{\nquotient}[2]{\ensuremath{\frac{\pi_1(#1)}{\pi_1(#1)^{(#2)}}}}
\newcommand{\nmquotient}[3]{\ensuremath{\frac{\pi_1(#1)^{(#2)}}{\pi_1(#1)^{(#3)}}}}
\newcommand{\pquotient}[2]{\ensuremath{\frac{\pi_1(#1)}{\pi_1(#1)_p^{(2)}}}}
\newcommand{\onepquotient}[2]{\ensuremath{\frac{\pi_1(#1)^{(1)}}{\pi_1(#1)_p^{(2)}}}}\newcommand{\iterate}[3]{\ensuremath{\underset{#2}{\overset{#3}{#1}}}}
\newcommand{\bdry}{\ensuremath{\partial}}

\begin{document}

\title[$\rho^1$ as a torsion obstruction.]{Von Neumann rho invariants and torsion in the topological knot concordance group}
\author{Christopher William Davis}

\address{Department of Mathematics \\ Rice University}
\email{cwd1@rice.edu}

\date{April 12, 2012}



\begin{abstract}    

We discuss an infinite class of metabelian Von Neumann $\rho$-invariants.  Each one is a homomorphism from the monoid of knots to $\R$.  In general they are not well defined on the concordance group.  Nonetheless, we show that they pass to well defined homomorphisms from the subgroup of the concordance group generated by anisotropic knots.  Thus, the computation of even one of these invariants can be used to conclude that a knot is of infinite order.  We introduce a method to give a computable bound on these $\rho$-invariants.  Finally we compute this bound to get a new and explicit infinite set of twist knots which is linearly independent in the concordance group and whose every member is of algebraic order 2.

\end{abstract}

\maketitle


\section{Introduction}\label{introduction}

In this paper we  study a particular class of von Neumann $\rho$-invariants and show that they provide a concordance obstruction.  We show that these invariants are particularly computable for knots of finite algebraic order and use them to establish a new linearly independent family of twist knots.  

A knot $K$ is an isotopy class of oriented locally flat embeddings of the circle $S^1$ into the 3-sphere $S^3$.  A pair of knots, $K$ and $J$, are called \textbf{topologically concordant} if there is a locally flat embedding of the annulus $S^1\times[0,1]$ into $S^3\times[0,1]$ mapping $S^1\times\{1\}$ to a representative of $K$ in $S^3\times\{1\}$ and $S^1\times\{0\}$ to a representative of $J$ {with its orientation reversed} in $S^3\times\{0\}$.  A knot is called \textbf{slice} if it is concordant to the unknot or equivalently if it is the boundary of a locally flat embedding of the 2-ball $B^2$ into the 4-ball $B^4$.  

The set of all knots has the structure of a commutative monoid with identity given by the unknot under the binary operation of connected sum.  This monoid is not a group.  The only knot with an inverse is the unknot.  The quotient by the equivalence relation given by concordance, however, is a group. The inverse of any knot $K$ is given by $-K$, the reverse of the mirror image of $K$.  This group is called the \textbf{topological concordance group} and is denoted $\C$.

Given a knot $K$ the \textbf{Alexander module} of $K$, denoted $A_0(K)$, is defined as the rational first homology of the universal abelian cover of the complement of the knot in $S^3$ or equivalently of $M(K)$, where $M(K)$ denotes zero surgery on $K$.  In the language of twisted coefficients, $A_0(K) = H_1(M(K);\Q[t^{\pm1}])$ where $t$ is the generator of the regular first homology of $M(K)$.  There is a nonsingular $\Q[t^{\pm1}]$-bilinear form \begin{equation*}Bl:A_0(K)\times A_0(K)\to \Q(t)/\Q[t^{\pm1}]\end{equation*} called the \textbf{Blanchfield linking form}.  

For a submodule $P$ of $A_0(K)$, the \textbf{orthogonal complement} of $P$ with respect to this bilinear form, denoted $P^\perp$, is given by the set of elements of $A_0(K)$ which annihilate $P$.  That is, \begin{equation*}P^\perp = \{q\in A_0(K)|Bl(p,q)=0\text{ for all }p\in P\}.\end{equation*}  $P$ is called \textbf{isotropic} if $P\subseteq P^\perp$ and is called \textbf{Lagrangian} or \textbf{self annihilating} if $P=P^\perp$.  We call $K$ \textbf{anisotropic} if $A_0(K)$ has no nontrivial isotropic submodules.  On the opposite end of the spectrum, a knot is called \textbf{algebraically slice} if it has a Lagrangian submodule.

The quotient of $\C$ by algebraically slice knots is called the \textbf{algebraic concordance group}.  It is shown in \cite{L5} that this quotient is isomorphic to $\Z^\infty \oplus \Z_2^\infty \oplus \Z_4^\infty$.  In particular, this shows that the concordance group has infinite rank.  There is, however, much more to the concordance group. For example,  in \cite[Theorem 5.1]{CG1} Casson and Gordon define a family of invariants and use them to show that, of the algebraically slice twist knots, only the unknot and the $-2$ twist knot (the stevedore knot) are slice.  As a consequence of their work Jiang \cite{Ji1} shows that there is an infinite set of algebraically slice twist knots that are linearly independent in $\C$.  Since then, the so-called Casson-Gordon invariants have served as useful tools in the detection of non-slice knots.

\begin{figure}[b]
\setlength{\unitlength}{1pt}
\begin{picture}(100,100)
\put(0,0){\includegraphics[width=.25\textwidth]{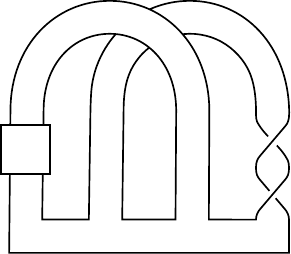}}
\put(5,30){$n$}
\end{picture}
\caption{the $n$-twist knot.}\label{fig:twist}
\end{figure}

Given any closed oriented 3-manifold $M$ and a homomorphism $\phi:\pi_1(M) \to \Gamma$, the Von Neumann $\rho$-invariant, $\rho(M, \phi)\in \R$, is defined.  It is an invariant of orientation preserving homeomorphism of the pair $(M,\phi)$.  Restricting this invariant to the zero surgery of knots and links gives rise to an isotopy invariant.

In \cite{whitneytowers}, Cochran, Orr and Teichner use $\rho$-invariants to show that there is an infinite rank subgroup of $\C$ of algebraically slice knots on which the Casson-Gordon invariants vanish.

In Section~\ref{easy}, we define the Von Neumann $\rho$-invariants in which we are interested.  Briefly, $\rho^0$ is the invariant associated with abelianization, $\rho^1$ with the quotient by the second term in the derived series and $\rho^1_p$ with $A_0^p(K)$, a localization of the Alexander module.

We proceed to give an additivity theorem for these invariants (Theorem~\ref{homomorphism}).  As a consequence, $\rho^1_p$ is a homomorphism from the monoid of knots to $\R$.  From the same theorem we deduce that there exist slice knots with non-vanishing $\rho^1_p$.  The existence of such knots implies that these invariants are not well defined on the concordance group.  In Section~\ref{rho as obstruction} we find that regardless of their ill-definedness there is a setting in which these $\rho$-invariants provide concordance information.  Theorem~\ref{big theorem 1} below is proven.  A Laurent polynomial $p(t)$ is called \textbf{symmetric} if $p(t)=t^k p(t^{-1})$ for some integer $k$.  
  
  \begin{reptheorem}{big theorem 1}
  If $p$ is a symmetric prime polynomial, $m_1,\dots, m_n\in \Z$, $K_1, \dots K_n$ are knots, $A_0^p(K_i)$ has no isotropic submodules for each $i$ and $\iterate{\#}{i=1}{n}m_iK_i$ is slice then $\displaystyle \sum_{i=1}^{n}m_i\rho^1_p(K_i)=0$.  
  \end{reptheorem}
  
  Restricting that theorem to the case that $n=1$ we get an application to the obstruction of torsion in the concordance group.
  
  \begin{repcorollary}{torsion corollary 1}
  Let $p$ be a symmetric polynomial and $K$ be a knot with $A_0^p(K)$ having no isotropic submodules.  If $K$ is of finite order in the concordance group then $\rho^1_p(K)=0$.  
  \end{repcorollary}
  
  In fact, Theorem~\ref{big theorem 1} implies that for each symmetric polynomial $p$ the $\rho^1_p$-invariant passes to a homomorphism from the subgroup of the concordance group generated by knots for which $A_0^p$ has no isotropic submodules.  For each $p$ this subgroup contains all knots with prime Alexander polynomials.  Thus, these invariants each give obstructions to linear dependence amongst these knots.

While Theorem~\ref{big theorem 1} can be compared to the results of \cite{derivatives}, it has the advantage that a knot can potentially be shown to be of infinite order after a single computation, while the results of \cite{derivatives} (for example Theorem 4.2 of that paper) only conclude that a knot is not slice, saying nothing about its concordance order and even then a $\rho$-invariant must be computed for every isotropic submodule, of which there may be many.


An unfortunate fact about these invariants, $\rho$-invariants in general and Casson-Gordon invariants is that it is difficult to do significant computations involving them.  Section~\ref{computational tools} addresses this shortcoming for our invariants by proving Theorem~\ref{premain} which relates the $\rho^1_p$-invariant of a knot of finite algebraic order with the $\rho^0$-invariant of any link representing a metabolizer of a Seifert form.

\begin{reptheorem}{premain}
Let $p$ be a symmetric prime polynomial.  Let $K$ be a knot of finite algebraic order $n>1$ with $A_0^p(K)$ having no isotropic submodules.  Let $\Sigma$ be a genus $g$ Seifert surface for $nK=\underset{n}{\#}K$.  Let $L$ be a link of $g$ curves on $\Sigma$ representing a metabolizer for the Seifert form.  Let $P$ be the submodule of $A_0^p(nK)$ generated by $L$.  Suppose that meridians about the bands on which the components of $L$ sit form a $\Z$-linearly independent set in ${A^{p}_0\left(nK\right)}/{P}$.   

Then, $ \left| n\rho^1_p(K)-\rho^0(L)\right| \le g-1.$
\end{reptheorem}

This is good news because $\rho^0$ is computable.  In many cases it can be expressed in terms of the integral of a simple function on $\T^n$, the $n$-dimensional torus.  The application of Theorem~\ref{big theorem 1} together with Theorem~\ref{premain} gives a new and tractable strategy to show that knots which are of finite algebraic concordance order are not of finite topological concordance order.  By finding a metabolizing link $L$ for $K$ and computing $\rho^0(L)$ one can hope to conclude that $K$ is of infinite concordance order.
 
This strategy is employed in Section~\ref{twist} to give a new infinite linearly independent set of twist knots whose every member is of order 2 in the algebraic concordance group.  Specifically, the theorem below is proven.  This appears to be the first application of von Neumann $\rho$-invariants to the twist knots of finite algebraic order.

\begin{reptheorem}{twist theorem}
For $x$ an integer, let $n(x) = -x^2-x-1$.  For $x\ge2$, the set containing the $n(x)$-twist knots, $\{T_{-7}, T_{-13},T_{-21} \dots\}$, is linearly independent in $\C$.  
\end{reptheorem}


In related work, a similar set, neither containing nor contained by the one presented here, is given by Tamulis \cite[corollary 1.2]{Tamulis}.  In \cite{knotconcordanceandtorsion}, Livingston and Naik find that twist knots which are of algebraic order 4 are of infinite concordance order.  In \cite{LiN1} they find many large families of algebraic order four twist knots are linearly independent.  In \cite{polynomialSplittingOfCG} Kim uses results of Gilmer to prove that except for the unknot, the $-1$-twist knot and the $-2$-twist knot, no nontrivial linear combination of twist knots is ribbon.  {In \cite[Corollary 1.3]{Lis1}, Lisca establishes the smooth concordance order of the twist knots (and of two-bridge knots in general). In the topological setting the concordance order of the twist knots is in general unknown, to say nothing of their linear independence.}

\section{Background information: basic properties of $\rho$ and $\sigma^{(2)}$}

For the definition of the Von Neumann rho invariant see for instance \cite[equation 2.10, definition 2.11]{CT} and \cite[section 3]{Ha2}.  For the definition of the $L^2$ signature see for example \cite[section 3.4, definition 3.21]{L2sign}.  Instead of presenting definitions, we give the properties needed for our analysis.

The first property we essentially employ as the definition of the $\rho$-invariant.  We even label it as such.  The pair $(W,\Lambda)$ is referred to in \cite{Ha2} as a stable null-bordism for $\{(M_i,\Gamma_i)\}_{i=1}^n$ where it is proven that such a definition is independent of the stable null-bordism used.

\begin{definition}\label{rho}
Consider oriented 3-manifolds $M_1, \dots, M_n$, with homomorphisms $\phi_i:\pi_1(M_i)\to \Gamma_i$.  Suppose that $M_1\sqcup M_2\sqcup\dots\sqcup M_n$ is the oriented boundary of a compact 4-manifold $W$ and $\psi:\pi_1(W)\to \Lambda$ is a homomorphism such that, for each $i$, there is a monomorphism $\alpha_i:\Gamma_i\to \Lambda$ making the following diagram commute:
\begin{center}$
 \begin{diagram}
\node{\pi_1(M_i)} \arrow{e,t}{\phi_i}
         \arrow{s,r}{i_*}
	     \node{\Gamma_i} \arrow{s,r,J}{\alpha_i}\\
       \node{\pi_1(W)} \arrow{e,t}{\psi}
     \node{\Lambda} 
      \end{diagram}$
      \end{center}

Then $\displaystyle\sum_{i=1}^n\rho(M_i,\phi_i) = \sigma^{(2)}(W,\psi) - \sigma(W)$ where $\sigma$ is the regular signature of $W$ and $\sigma^{(2)}$ is the $L^2$ signature of $W$ twisted by the coefficient system $\psi$.
\end{definition}

The primary tool in this paper for getting information about the $L^2$ signature of a 4-manifold is a bound in terms of the rank of twisted second homology.  When $\Gamma$ is PTFA (Poly-Torsion-Free-Abelian, see \cite[Definition 2.1]{whitneytowers}) and more generally whenever $\Q[\Gamma]$ is an Ore domain,
\begin{equation}\label{inequality}
\left|\sigma^{(2)}(W,\phi)\right| \le \rank_{\Q[\Gamma]}\left(\dfrac{H_2\left(W;\Q[\Gamma]\right)}{i_*\left[H_2\left(\bdry W;\Q[\Gamma]\right)\right]}\right),
\end{equation}
where $i_*:H_2\left(\bdry W;\Q[\Gamma]\right)\to H_2\left(W;\Q[\Gamma]\right)$ is the inclusion induced map.  

This follows from the monotonicity of von Neumann dimension (see \cite[Lemma 1.4]{L2invts}) and the fact that $L^2$ Betti number agrees with $\Q[\Gamma]$ rank when $\Q[\Gamma]$ is an Ore Domain. (see \cite [Lemma 2.4] {Cha3} or \cite[Proposition 2.4]{FrLM}).

\section{The invariants of interest and some easy results}\label{easy}

In this section we define the invariants studied in this paper.  Each is the Von Neumann $\rho$-invariant with respect to some abelian or metabelian quotient of the fundamental group of zero surgery on a knot or (in the case of $\rho^0$) a link.

\begin{definition}\label{rho0}
For a link $L$ of $n$ components with zero linking numbers, let $\phi^0:\pi_1(M(L))\to \Z^n$ be the abelianization map.  Let $\rho^0(L) = \rho(M(L), \phi^0)$ be the corresponding $\rho$-invariant.
\end{definition}

For a knot $K$, $\rho^0(K)$ is equal to the integral of the Levine-Tristram signature function (see \cite[Proposition 5.1]{structureInConcordance}).  In particular, this invariant is computable but is zero for every knot of finite order in the algebraic concordance group.  Despite this, $\rho^0$ can be used to get concordance information about algebraically slice knots by studying $\rho^0(L)$ where $L$ is a link given by a metabolizer of the knot.  In \cite[4.3]{Collins} and \cite[example 5.10]{derivatives}, this strategy is used to give an alternate proof of the result of Casson-Gordon \cite[Theorem 5.1]{CG1} that the algebraically slice twist knots are not slice.  In this paper the $\rho^0$-invariant of a metabolizing link is used to get information about the invariants given in the next two definitions.

\begin{definition}\label{rho1}
For a knot $K$ let $\phi^1:\pi_1(M(K))\to \nquotient{M(K)}{2}$ be the projection map.  Let $\rho^1(K) = \rho\left(M(K),\phi^1 \right)$ be the corresponding $\rho$-invariant.
\end{definition}

In order to give the definition of the third invariant we must first provide some definitions involving localizations of the Alexander module of a knot.  For $p(t)\in \Q[t^{\pm1}]$ let \begin{equation}R_p = \left.\left\{\frac{f}{g}\right|(g,p)=1\right\}\end{equation} be the \textbf{localization of ${\Q[t^{\pm1}]}$ at ${p}$}.  For a knot, $K$, let $A_0^p(K) = A_0(K)\otimes R_p$ be the \textbf{localization of the Alexander module of ${K}$ at ${p}$.}  (The usual notation for localization would call this the localization at the multiplicative set of polynomials relatively prime to $p$.)

Throughout this paper we need to be flexible with notation.  For any CW complex $X$ with $H_1(X)=\langle t\rangle\cong \Z$ we define the Alexander module of $X$, $A_0(X)$, as the homology of the universal abelian cover of $X$.  Similarly, we define the localized Alexander modules of $X$, $A_0^p(X)=A_0(X){\otimes} R_p$.  In this paper, such an $X$, if not zero surgery on a knot, is generally a 4-manifold cobordism between zero surgeries on knots.

\begin{definition}\label{rho1p}
For a polynomial $p$ and a knot $K$ let $\pi_1(M(K))^{(2)}_p$ be the kernel of the composition
\begin{equation*}
\pi_1(M(K))^{(1)}\to\nmquotient{M(K)}{1}{2}\hookrightarrow A_0(K)\to A_0^p(K).
\end{equation*}
 Let $\phi^1_p:\pi_1(M(K))\to\pquotient{M(K)}{p}$ be the projection map and $\rho^1_p(K)=\rho\left(M(K), \phi^1_p \right)$ be the $\rho$-invariant associated to this homomorphism.
\end{definition}

These $\rho$-invariants are similar to the first order $\rho$-invariants defined and employed in \cite{derivatives}.  One could view our definition of $\rho^1$ as the restriction of their invariants to the case of anisotropic Alexander modules (a setting not considered in that paper). The $\rho^1_p$-invariant does not appear to have been previously considered.

The following theorem describes interactions between these invariants.  It can be thought of as suggesting that the $\rho^1_p$-invariant picks up information sitting in the $p$-torsion part of the Alexander module of the knot.  

\begin{proposition}\label{rho prime}
\begin{enumerate}
Let $\Delta$ be the Alexander polynomial of a knot $K$.
\item
If $p$ is a polynomial relatively prime to $\Delta$ then $\rho^1_p(K) = \rho^0(K)$.  
\item
If $\Delta = p$ then $\rho^1_p(K)=\rho^1(K)$. 
\end{enumerate}
\end{proposition}
\begin{proof}
When $p$ is relatively prime to the Alexander polynomial of $K$ then $\Delta\in S_p$ is invertible in $R_p$.  Since $\Delta$ annihilates $A_0(K)$, this implies that $A_0^p(K)=A_0\otimes Q[t^{\pm1}]S_p^{-1}$ is the trivial module so the kernel of 
\begin{equation*}
\pi_1(M(K))^{(1)}\to\nmquotient{M(K)}{1}{2}\hookrightarrow A_0(K)\to A_0^p(K)=0
\end{equation*}
is equal to $\pi_1(M(K))^{(1)}$.  Thus, the map $\phi^1_p$ of definition~\ref{rho1p} is the same as the map $\phi^0$ of definition~\ref{rho0} (the abelianization map).  This completes the proof of the first claim.

When $p$ is equal to the Alexander polynomial of $K$ the map $A_0(K)\to A_0^p(K)$ is injective.  Thus, the kernel of
\begin{equation*}
\pi_1(M(K))^{(1)}\to\nmquotient{M(K)}{1}{2}\hookrightarrow A_0(K)\hookrightarrow A_0^p(K)
\end{equation*}
is equal to $\pi_1(M(K))^{(2)}$ and $\phi^1_p$ is exactly the map $\phi^1$ of definition~\ref{rho1}.  This completes the proof of the second part.

%
%
\end{proof}

 Throughout this paper we make use of a pair of additivity properties. The second is a localized version of \cite[Theorem 11.1 parts 4, 5]{C}.  The knot $J_\eta(K)$ is the result of infection.  For an overview of infection, see \cite[section 8]{C}.

\begin{proposition}\label{homomorphism}
Let $J$ and $K$ be knots and $\eta$ be an unknot in the complement of $J$ such that the linking number between $J$ and $\eta$ is zero.  Let $p$ be a polynomial.

\begin{enumerate}
\item $\rho^1_p(J\# K) = \rho^1_p(J) + \rho^1_p(K)$
\item $\displaystyle \rho^1_p(J_\eta(K)) = \left\{ 
\begin{array}{ccc}
\rho^1_p(J) & \text{if} & \eta=0 \text{ in } A_0^p(J)\\
\rho^1_p(J)+\rho^0(K) & \text{if} & \eta\neq0 \text{ in } A_0^p(J)\\
\end{array}
\right.$
\end{enumerate}
\end{proposition}
\begin{proof}

The proof of part 1 proceeds by constructing a 4-manifold $W$ with $\bdry(W) = M(J)\sqcup M(K)\sqcup -M(J\#K)$ such that:
\begin{enumerate}
\item
 The map induced by the inclusion of each boundary component of $W$ on the first homology is an isomorphism.  We let $t$ be the generator of first homology of $W$ and of of every one of its boundary components.  
 \item
  $H_1(W;R_p) \cong H_1(M(J);R_p)\oplus H_1(M(K);R_p)$ and the inclusion induced maps from $H_1(M(J);R_p)$ and $H_1(M(K);R_p)$ are the maps to the first and second factors of this direct sum.
   \item
  $H_1(M(J\#K);R_p) \cong H_1(W;R_p)$ and the inclusion induced map is an isomorphism.
  
  \item
  $H_2(W;\Z)$ is carried by $\bdry W$.
  \item
  $H_2(W;\K)=0$, where $\K$ is the classical field of fractions of the Ore domain $\Q\left[\pquotient{W}{p}\right]$.
  \end{enumerate}
  
  Supposing that such a $W$ is found, the next step is to show that the inclusion induced map $\pquotient{M(K)}{p}\to \pquotient{W}{p}$ is injective for each boundary component.  Since we reuse this idea with some frequency throughout this paper, we write it down as a lemma:
  
  \begin{lemma}\label{repeated use lemma}
  Let $p$ be a polynomial.  Suppose that $M(K)$ is a boundary component of a 4-manifold $W$, that the inclusion induced map $H_1(M(K))\to H_1(W)$ is an isomorphism and that the inclusion induced map $i_*:A_0^p(K)\to A_0^p(W)$ is injective.
  
  Then the inclusion induced map $\pquotient{M(K)}{p}\to \pquotient{W}{p}$ is injective.
  \end{lemma}
  \begin{proof}[Proof of Lemma~\ref{repeated use lemma}]
  The following diagram commutes and has exact rows
  \begin{equation}\label{lemma-diagram}
  \begin{diagram}\dgARROWLENGTH=1.5em 
  \node{0}\arrow{e}
  \node{\onepquotient{M(K)}{p}} \arrow{e}\arrow{s,r}{\beta}
     \node{\pquotient{M(K)}{p}} \arrow{e} \arrow{s,r}{i_*}
     \node{\nquotient{M(K)}{1}}\arrow{s,r}{\cong}\arrow{e}
     \node{0}\\
     \node{0}\arrow{e}
     \node{\onepquotient{W}{p}} \arrow{e}
     \node{\pquotient{W}{p}} \arrow{e}
     \node{\nquotient{W}{1}}\arrow{e}
     \node{0}
      \end{diagram}
      \end{equation}

  The vertical map on the right is exactly the inclusion induced map on first homology which is an isomorphism by assumption.  If $\beta$ is a monomorphism, then $i_*$ is a monomorphism.
  
  Consider the commutative diagram
  \begin{equation}\begin{diagram}
  \node{\onepquotient{M(K)}{p}} \arrow{e,J}\arrow{s,l}{\beta}
     \node{A_0^p(M)}\arrow{s,J}\\
     \node{\onepquotient{W}{p}} \arrow{e,J}
     \node{A_0^p(W)}
      \end{diagram}
      \end{equation}  
      The two horizontal maps are injective by the definition of $\pi_1(M)^{(2)}_p$, while the rightmost vertical map is injective by assumption.  Thus, $\beta$ is injective and the proof is complete.

  \end{proof}
  
%
      
      Applying Lemma~\ref{repeated use lemma} for each boundary component and using properties (1), (2) and (3), the conditions of definition~\ref{rho} are satisfied, so \begin{equation}\rho^1_p(M(K\#J))-\rho^1_p(K)-\rho^1_p(J) = \sigma^{(2)}\left(W;\phi\right)-\sigma(W),\end{equation}
where $\phi:\pi_1(W)\to\pquotient{W}{p}$ is the quotient map.
  By property (4), $\sigma(W)=0$.  By property (5) and inequality~\eqref{inequality}, $\sigma^{(2)}\left(W;\phi\right)=0$.  It remains only to construct such a $W$.
  
    Construct $W$ by taking $$M(J) \times [0,1]\sqcup M(K)\times[0,1]$$ and connecting it by gluing together neighborhoods of curves in $M(J)\times\{1\}$ and $M(K)\times\{1\}$ representing the meridians of $J$ and $K$ respectively.  This could equally well be described via the addition of a 1-handle and a 2-handle to $M(J) \times [0,1]\sqcup M(K)\times[0,1]$.
    
  
  
 It can be seen that $\bdry W=M(J)\sqcup M(K) \sqcup -M(K\#J)$ by thinking of connected sum as infection along a meridian.  We now provide the proof that $W$ has properties (1) through (5).    
  
Consider the last 4 terms of the long exact sequence of the pair $(W,M(J)\sqcup M(K))$:
\begin{equation}\label{end of sequence} 
H_2(W,M(J)\sqcup M(K))\to H_1(M(J)\sqcup M(K))\to H_1(W)\to 0.
\end{equation}
Letting $e$ be the relative second homology class given by core of the added $S^1\times B^2\times [0,1]$ and $\mu_J$, $\mu_K$ be the meridians of $K$ and $J$ respectively, this exact sequence becomes
\begin{equation}\label{first few terms}
\begin{diagram}
\node{\langle e\rangle}\arrow{e,b}{e\mapsto \mu_J-\mu_K}\node{\langle\mu_J, \mu_k\rangle}\arrow{e}\node{ H_1(W)}\arrow{e}\node{ 0}
\end{diagram}
\end{equation}
Thus, $H_1(W)=\Z$ is generated by either $\mu_J$ or $\mu_K$, so the inclusion from $M(K)$ and $M(J)$ are both isomorphisms on first homology.  The meridian of $M(K\#J)$ is isotopic in $W$ to $\mu_J$ so the map from $H_1(M(K\#J))$ to $H_1(W)$ is also an isomorphism.  This proves claim (1).  

Consider at the terms previous to (\ref{end of sequence}) in the long exact sequence:
\begin{equation}\label{next few terms}
\begin{array}{c}H_2(M(J)\sqcup M(K))\overset{i_*}{\to} H_2(W)\overset{p_*}{\to} H_2(W,M(J)\sqcup M(K))\\\overset{\bdry_*}{\to} H_1(M(J)\sqcup M(K))\end{array}
\end{equation}
We saw in (\ref{first few terms}) that $\bdry_*$ in (\ref{next few terms}) is injective, so $p_*$ is the zero map and \begin{equation}H_2(M(J)\sqcup M(K))\overset{i_*}{\to} H_2(W)\end{equation} is an epimorphism, proving (4).

Consider the decomposition of $W$ as $W=M(K)\times [0,1]\cup M(J)\times [0,1]$.  The intersection $M(K)\cap M(J)=S^1\times B^2$ is a neighborhood of the meridian of $J$ and $K$.  The Mayer Vietoris sequence corresponding to this decomposition with coefficients in $R_p$ is 
\begin{equation}\label{MVS}
\begin{array}{c}
H_n(S^1\times B^2;R_p)\overset{i_*}{\to} H_n(M(K);R_p)\oplus H_n(M(J);R_p) \overset{j_*}{\to} \\H_n(W;R_p)\overset{\bdry_*}{\to} H_{n-1}(S^1\times B^2;R_p).
\end{array}
\end{equation}
The cover of $S^1\times B^2$ corresponding to this coefficient system is given by $\R\times B^2$.  This is contractible so $H_n(S^1\times B^2;R_p)=0$.

Thus, for each $n$, the map 
\begin{equation}\label{isom}
H_n(M(K);R_p)\oplus H_n(M(J);R_p) \to H_n(W;R_p)
\end{equation}
is an isomorphism.  Setting $n=1$ proves (2).

The isomorphism exhibited in (\ref{isom}) 
implies that $H_n(W,M(K)\sqcup M(J);R_p)=0$ for all $n$.  By Poincar\'e Duality, then $H^n(W,M(K\#J);R_p)=0$ for all $n$.  By the universal coefficient theorem \cite[Theorem 2.36]{DaKi} $H_n(W,M(K\#J);R_p)=0$.  An examination of the the long exact sequence of the pair $(W,M(K\#J))$ reveals that the inclusion induced map from $H_n(M(K\#J);R_p)$ to $H_n(W;R_p)$ is an isomorphism for all $n$.  Taking $n=1$ proves (3).

Now consider the Mayer Vietoris sequence in \eqref{MVS} with coefficients in $\K$ instead of $R_p$.  
\begin{equation}\label{MVSK}
\begin{array}{c}
H_n(S^1\times B^2;\K)\to H_n(M(K);\K)\oplus H_n(M(J);\K)\to\\ H_n(W;\K)\to H_{n-1}(S^1\times B^2;\K).
\end{array}
\end{equation}
By \cite[Corollary 3.12]{C} $H_n(S^1\times B^2;\K)$, $H_n(M(K);\K)$ and $H_n(M(J);\K)$ vanish for all $n$, implying that $H_n(W;\K)=0$, which together with the inequality~(\ref{bound}) proves (5) and completes the proof of part 1 of the theorem.

  The second part of the theorem can be proved in an analogous manner.  We give a description of $W$ and allow the interested reader to work out details.  
  
    Let $W$ be given by taking $M(J) \times [0,1]\sqcup M(K)\times [0,1] $ and gluing  a neighborhood of $\eta$ in $M(J)\times\{1\}$ to a neighborhood of the meridian if $M(K)\times \{1\}$.  The boundary of $W$ is given by $M(J) \sqcup M(K) \sqcup -M\left(J_\eta(K)\right)$.
    
    Similar properties are desired.  Specifically:
    \begin{enumerate}
    \item[(1')] The inclusion induced maps from $H_1(M(J))$ and $H_1(M(J_\eta(K)))$ to $H_1(W)$ are both isomorphisms.  The inclusion induced map from $H_1(M(J))$ to $H_1(W)$ is the zero map.  Let $t$ be the generator of $H_1(W)$.
    \item[(2')] The inclusion induced maps from $H_1(M(J);R_p)$ and $H_1(M(J_\eta(K));R_p)$ to $H_1(W;R_p)$ are both isomorphisms.  
    \item[(3')]
    The composition 
    \begin{equation*}
    H_1(M(J))\to H_1(M(J))\underset{\Z}{\otimes} R_p = H_1(M(J);R_p) \to H_1(W;R_p)
    \end{equation*} is the zero map if $\eta$ is zero and otherwise is injective.
  \item[(4)]
  $H_2(W;\Q)$ is carried by $\bdry W$.
  \item[(5)]
  $H_2(W;\K)=0$, where $\K$ is the classical field of fractions of the Ore domain $\Q\left[\pquotient{W}{p}\right]$.
    \end{enumerate} 
  \end{proof}
  
Notice that the first part of the Theorem~\ref{homomorphism} states that each of these metabelian $\rho$-invariants is a homomorphism from the monoid of knots to $\R$.  One might hope that they pass to homomorphisms on the knot concordance group.  This is not the case.  Consider the pair of slice knots depicted in Figure~\ref{fig:counterexample}.  If $\rho^0(K)\neq 0$, for example if $K$ is a trefoil knot, then $\rho^1(R)$ and $\rho^1(R_\eta(K))$ cannot both be zero by the second part of Proposition~\ref{homomorphism}. 

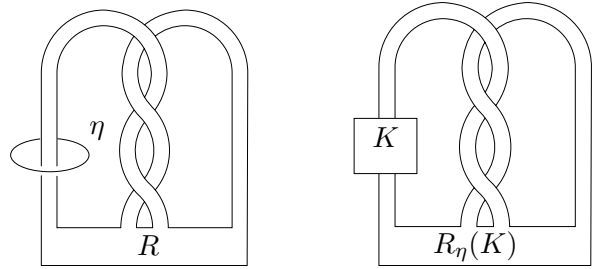
\begin{figure}[h]
\setlength{\unitlength}{1pt}
\begin{picture}(220,100)
\put(00,0){\includegraphics[width=.25\textwidth]{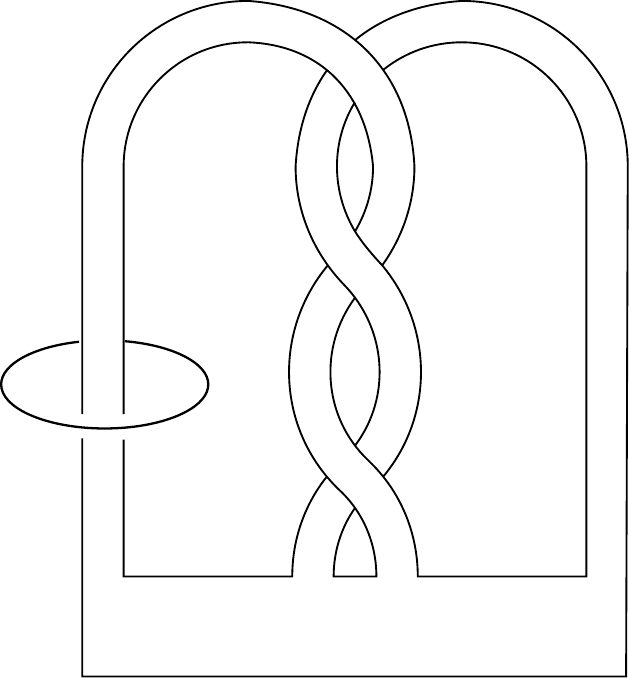}}
\put(130,0){\includegraphics[width=.25\textwidth]{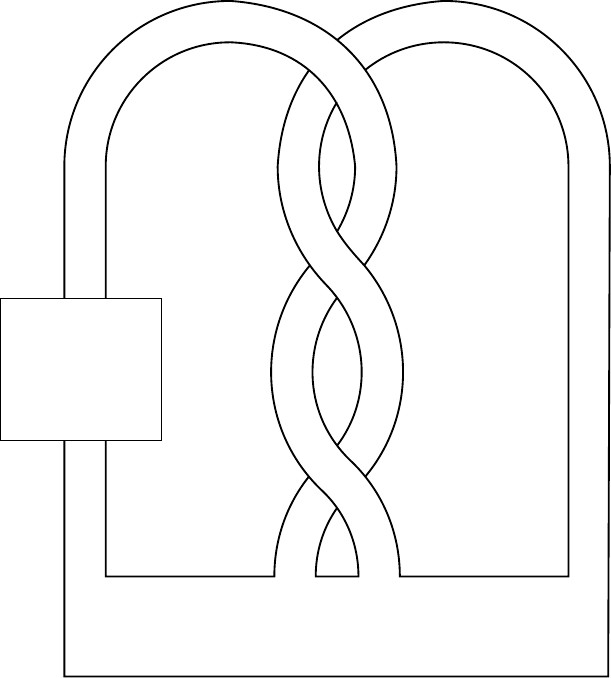}}
\put(30,50){$\eta$}
\put(137,45){$K$}
\put(48,4){$R$}
\put(160,5){$R_\eta(K)$}
\end{picture}
\caption{A pair of slice knots with differing $\rho^1$-invariant.}\label{fig:counterexample}
\end{figure}

This illustrates two obstacles to the use of $\rho^1_p$ as a concordance tool.  The first is that it is not well defined on the concordance group.  The next section, despite this fact, gets concordance information from these invariants.  The second obstacle is the well known difficulty of actually computing $L^2$ signatures.  As a consequence, in the example above we do not give a specific slice knot with non-zero $\rho^1$.  Rather we gave a pair, at least one of which has non-zero $\rho^1$.

  In Section~\ref{computational tools} we present a solution to the second obstacle by finding computable bounds on $\rho^1$.  We compute these bounds in Section~\ref{twist} for a family of twist knots.  

\section{A context in which $\rho^1_p$ provides concordance information}\label{rho as obstruction}

The ring $\Q[t^{\pm1}]$ has an involution defined by 
  \begin{equation}
  q(t)\mapsto\overline{q}(t)=q(t^{-1}).
  \end{equation}  
  When $p(t)$ is symmetric, this involution extends to the localization $R_p$.  For any knot $K$, the classical Blanchfield form $Bl$ (discussed in the introduction to this paper) extends to a sesquilinear form \begin{equation}Bl:A_0^p(K)\times A_0^p(K)\to\dfrac{\Q(t)}{R_p}.\end{equation}  
  
  Recall the definitions of isotropic and Lagrangian submodules of the Alexander module given in Section~\ref{introduction}.  The definitions of these concepts are identical in the localized Alexander module.  A knot is called $p$-anisotropic if its localized Alexander module, $A_0^p(K)$, has no nontrivial isotropic submodules.

%

The goal of this section is a theorem which allows the use of $\rho^1_p$ to obtain concordance information for knots whose every prime factor is $p$-anisotropic. One can think of  this restriction as allowing us to use the mindset in \cite{derivatives} without having to think about $\rho$-invariants corresponding  to the (possibly infinitely) many  isotropic submodules of the Alexander module.

\begin{theorem}\label{big theorem 1}
Given a symmetric polynomial $p$, integers $m_1,\dots m_n$ and $p$-anisotropic knots $K_1, \dots K_n$, if $m_1K_1\#\dots \# m_nK_n$ is slice then $\displaystyle \sum_{i=1}^n m_i\rho^1_{p}(K_i) = 0$
\end{theorem}

The proof is delayed to subsection~\ref{proof of big 1}.  We begin with an exploration of its implications as a means of detecting linear dependence in the concordance group.  For example, taking $m_1=\dots= m_n$, it yields the following torsion obstruction:
\begin{corollary}\label{torsion corollary 1}
For a symmetric polynomial $p$ and a knot $K$ which decomposes into a connected sum of $p$-anisotropic knots,  if $K$ is of finite order in the concordance group, then $\rho^1_p (K) = 0$
\end{corollary}
 
 Since we are restricting our field of vision to $p$-anisotropic knots, it is worth noting how many knots are $p$-anisotropic.  The following theorem shows that there are many knots to which Theorem~\ref{big theorem 1} applies.   Its proof is analogous to the proof of  \cite[Theorems 4.1 through 4.3]{Go2}.

\begin{proposition}\label{applies}
For a knot $K$ and a polynomial $p$ which has no non-symmetric factors,  $A_0^p(K)$ is anisotropic if each factor of $p$ divides the Alexander polynomial of $K$ with multiplicity at most $1$.\end{proposition}
\begin{proof}[Proof of Proposition~\ref{applies}]
Let $\Delta_K$ be the Alexander polynomial of $K$.

$A_0(K)$, being a finitely generated torsion module over the PID $\Q[t^{\pm1}]$, has a decomposition into elementary factors:
\begin{equation}
A_0(K) = \iterate{\oplus}{i=1}{n} \dfrac{\Q[t^{\pm1}]}{(q_i)}
\end{equation}
where $q_i$ divides $q_{i+1}$ for each $i$ and $q_1 q_2\dots q_n=\Delta_K$.  Let $h$ be the greatest common divisor of $p$ and the Alexander polynomial of $K$.  Each prime factor $f$ of $h$ divides some $q_i$.  If $i<n$ then since $f$ also divides $q_{i+1}$, it follows that $f$ divides that Alexander polynomial of $K$ with multiplicity greater than $1$, contradicting the assumption to the contrary.  

Thus, for every $i<n$, $p$ is relatively prime to $q_i$, so $\frac{\Q[t^{\pm1}]}{(q_i)}\otimes R_p=0$ and 
\begin{equation}
A_0^p(K) = \dfrac{\Q[t^{\pm1}]}{(q_n)}\otimes R_p = \dfrac{R_p}{(q_n)}
\end{equation}
is cyclic.  Since $q_n=h h'$ where $h'$ is relative prime to $p$ and thus is a unit in $R_p$, the ideals $(q_n)$ and $(h)$ in $R_p$ are equal.  Thus,
\begin{equation}
A_0^p(K) = \dfrac{R_p}{(q_n)} = \dfrac{R_p}{(h)}
\end{equation}

Let $\eta$ be a generator of $A_0^p(K)$.  Then since the localized Blanchfield form is nonsingular, $Bl(\eta, \eta)=\frac{r}{h}\in\frac{\Q(t)}{R_p}$, where $r$ and $h$ are coprime.  

If $Bl(s(t)\eta, s(t)\eta) = \displaystyle\dfrac{s\overline{s}r}{h}=0$, that is, $s(t)\eta$ is an isotropic element, then it must be that $h$ divides $s\overline{s}$.  Since $p$ and thus $h$ are squarefree and have no non-symmetric factors, this implies that $h$ divides $s$ so $s(t)\eta=0$.  Thus, any element of an isotropic submodule of $A_0^p(K)$ must be the zero element and so $K$ is $p$-anisotropic.  

\end{proof}

Thus, Theorem~\ref{big theorem 1} applies for every choice of $p$ when the Alexander polynomial of $K_i$ is square-free for each $i$.

The restriction to the case of knots of finite algebraic concordance order and coprime squarefree Alexander polynomials can be addressed by the repeated application of Theorem~\ref{big theorem 1}.  Doing so yields the following corollary, which can be thought of as a style of polynomial splitting theorem as in \cite{polynomialSplittingOfRho}.

\begin{corollary}\label{big corollary 1}
Let $K_1, K_2,\dots$ be knots of finite algebraic order which have coprime squarefree Alexander polynomials.  If $\rho^1(K_i)$ is nonzero for each $i$, then the knots $K_i$ form a  linearly independent set in $\C$.
\end{corollary}
\begin{proof}

Let $p_n$ be the Alexander polynomial of $K_n$.  

Suppose that some linear combination $m_1K_1\#\dots \# m_kK_k$ is slice.  Proposition~\ref{applies} gives that $K_i$ is $p_n$-anisotropic for every $i$ so 
\begin{equation}\label{big corollary equation}
\displaystyle \sum_{i=1}^k m_i\rho^1_{p_n(t)}(K_i) = 0
\end{equation} 
by Theorem~\ref{big theorem 1}.  Proposition~\ref{rho prime} applies to give that $\rho^1_{p_n}(K_i)=\rho^0(K_i)$ when $i\neq n$ and $\rho^1_{p_n}(K_n)=\rho^1(K_n)\neq 0$.  Since $K_i$ is assumed to be of finite algebraic order, $\rho^0(K_i) = 0$.   

Thus, all but one of the terms on the left hand side of \ref{big corollary equation} are zero.  Dropping them reveals that $m_n\rho^1(K_n) = 0$.  Since $\rho^1(K_n)$ is assumed to be nonzero, it must be that $m_n$ is zero.  Repeating this argument for every natural number $n$ gives that these knots are linearly independent.
\end{proof}

\subsection{Examples}\label{subsect:examples}

As applications of Theorem~\ref{big theorem 1} we give some infinite lineary independent subsets of $\C$.

\begin{example}Consider any knot $J$ with non-zero $\rho^0$-invariant.  For every integer $n$, excluding those of the form $n=-x^2-x$ for $x\in \Z$, let the knot $K_n$, be given by connected sum of $T_n(J)$  (the $n$-twisted double of $J$) with $-T_n$ (the reverse of the mirror image of the $n$ twist knot).  In this example we show that the set containing all such $K_n$ is linearly independent in $\C$.  These knots are depicted in Figure~\ref{fig:linearly independent}.

\begin{figure}[h]
\setlength{\unitlength}{1pt}
\begin{picture}(165,70)
\put(00,0){\includegraphics[width=.45\textwidth]{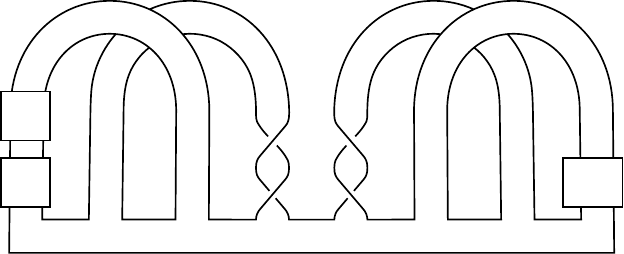}}
\put(04,15){$n$}
\put(146,16){{$-n$}}
\put(04,33){$J$}
\end{picture}
\caption{An infinite family of  algebraically slice knots which are linearly independent in $\C$. ($n\neq -x^2-x$)}\label{fig:linearly independent}
\end{figure}
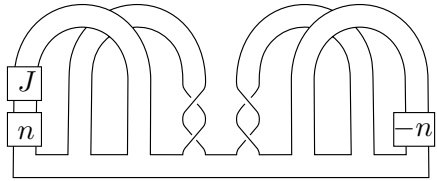

In order to see this, suppose $\iterate{\#}{n}{} c_n (T_n(J)\#-T_n)$ is slice.  By Proposition~\ref{applies}, the knots, $T_n$ and $T_n(J)$ are $p$-anisotropic for every symmetric polynomial $p$.  Thus, Theorem~\ref{big theorem 1} applies to give that 
\begin{equation}\label{e1}
\displaystyle \sum_n c_n\left(\rho^1_p(T_n(J))-\rho^1_p(T_n)\right)=0
\end{equation}
Take $p$ to be the the Alexander polynomial of $T_m$.  Since the twist knots all have distinct prime Alexander polynomials, Proposition~\ref{rho prime}, applies to reduce (\ref{e1}) to 
\begin{equation}\label{first example equation}
\displaystyle \sum_{n\neq m} \left(c_n\rho^0(T_n(J))-c_n\rho^0(T_n)\right)+c_m\rho^1(T_m(J))-c_m\rho^1(T_n)=0
\end{equation}
By Proposition~\ref{homomorphism} $c_m\rho^1(T_m(J))-c_m\rho^1(T_n)=c_m\rho^0(J)$.  Since $\rho^0$ depends only on the   algebraic concordance class of a knot, $\rho^0(T_n(J))=\rho^0(T_n)$.  Plugging these values into equation~\eqref{first example equation} yields
\begin{equation}\displaystyle c_m\rho^0(J)=0.\end{equation}
Since $\rho^0(J)$ is nonzero, it must be that $c_m$ is zero.  Repeating this argument for every natural number $m$ we see that these knots are linearly independent.

\end{example}

Notice that the knots $K_n$ in the preceding example are algebraically slice and in particular are not anisotropic.  Despite this fact Theorem~\ref{big theorem 1} applied to give a proof that they are linearly independent.  The $\rho^1_p$-invariant gives concordance information about knots sitting in the subgroup generated by $p$-anisotropic knots.  This group includes many $p$-isotropic knots and even some algebraically slice knots.

\begin{example}Consider the knot $J_n$ depicted on the left hand side of Figure~\ref{fig:linearly independent order 2}.  For $n>0$ it is anisotropic, having prime Alexander polynomial $$\Delta_{J_n}(t) = n^2t^2+(1-2n^2)t+n^2.$$ It is concordance order 2 so Corollary~\ref{torsion corollary 1} applies to give that $\rho^1(J_n)=0$ for all $n$.  For each $n$ pick a curve $\eta_n$  which is nonzero in $A_0(J_n)$.  Let $T$ be a knot with nonvanishing $\rho^0(T)$-invariant.  Let $K_n$ be the result of the infection depicted in Figure~\ref{fig:linearly independent order 2}.  By the second part of Lemma~\ref{homomorphism} $\rho^1(K_n) = \rho^0(T)$.  Corollary~\ref{big corollary 1} then applies to show that the set $\{K_n\}^\infty_{n=1}$ is linearly independent.  
\end{example}

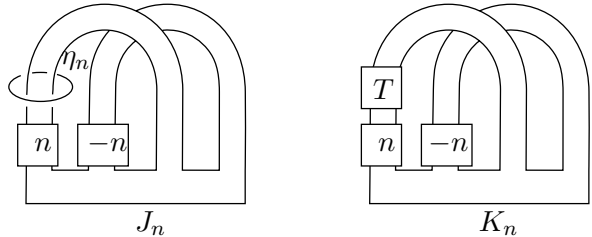
\begin{figure}[htbp]
\setlength{\unitlength}{1pt}
\begin{picture}(220,80)
\put(0,10){\includegraphics[width=.25\textwidth]{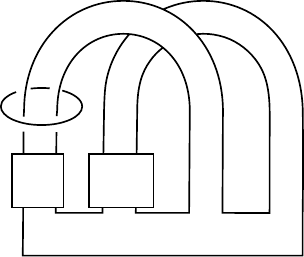}}
\put(130,10){\includegraphics[width=.25\textwidth]{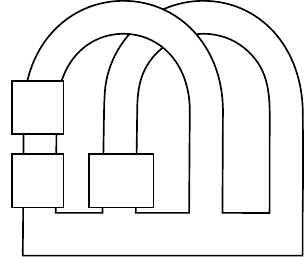}}
\put(10,30){$n$}
\put(30,30){$-n$}
\put(140,30){$n$}
\put(159,30){$-n$}
\put(20,63){$\eta_n$}
\put(138,50){$T$}
\put(48,0){$J_n$}
\put(178,0){$K_n$}
\end{picture}
\caption{Left: An infinite family of knots of order 2.  Right: An infinite family of knots which are algebraically of order 2 which is linearly independent in $\C$.  $n>0$.}\label{fig:linearly independent order 2}
\end{figure}

These examples both hinge on knowlege of the behavior of $\rho^1_p$ under infection.  If one wishes to say somthing about knots which do not result from infection, such as the twist knots, then another tool is needed.  Such a tool is found in Section~\ref{computational tools} and is employed to find information about twist knots in Section~\ref{twist}.

\subsection{The proof of Theorem~\ref{big theorem 1}}\label{proof of big 1}

We now prove Theorem~\ref{big theorem 1}.  The proof is  compartmentalized into several lemmas.

\begin{proof}[Proof of Theorem~\ref{big theorem 1}]

By replacing $K_i$ by $-K_i$ if nessessary we may assume that each $m_i$ is non-negative.  By replacing $m_iK_i$ by $K_i\#\dots \#K_i$ we may assume that $m_i=1$ for each $i$.   A 4-manifold $W$ is constructed whose boundary consists of $M(K_1)\sqcup\dots \sqcup M(K_n)$.  From here on $M(K_i)$ is abbreviated as $M_i$.  It will be shown using Lemma~\ref{repeated use lemma} as well as Lemmas~\ref{H1 isom} and \ref{kernel is isotropic} that the inclusion induced maps $$\alpha_i:\pquotient{M_i}{p}\to\pquotient{W}{p}$$ are monomorphisms.  This 4-manifold will be shown in Lemmas~\ref{second homology by bdry} and \ref{twisted homology} to have signature defect equal to zero with respect to the homomorphism, \begin{equation}\psi:\pi_1(W)\to\pquotient{W}{p}.\end{equation}  Once this is done, then by definition~\ref{rho}, \begin{equation}\displaystyle \sum_{i=1}^n m_i\rho^1_p(K_i) = 0.\end{equation}

      We let $V$ be the connected 4-manifold given by taking the disjoint union $\iterate{\sqcup}{i=1}{n} M_i\times [0,1]$ together with $n-1$ copies of $S^1\times B^2\times[0,1]$ indexed from $1$ to $n-1$ and gluing the i'th copy of $S^1\times B^2\times\{0\}$  to a neighborhood of a curve representing the meridian in $M_i\times\{1\}$ and the i'th copy of $S^1\times B^2\times\{1\}$ to a neighborhood of a curve representing the meridian in $M_{i+1}\times\{1\}$.  Let $\bdry_-(V) = \iterate{\sqcup}{i=1}{n}M_i$ denote the disjoint union of the $M_i$ boundary components of $V$.  Let $\bdry_+ V$ denote the $-M(K_1\#\dots \#K_n)$ boundary component of $V$.

      Suppose that $K:= K_1\#\dots \#K_n$ is slice.  Let $E$ be the complement of a slice disk for $K$ in $B^4$.  $\bdry E = M(K) = -\bdry_+V$.  Let $W$ be the union of $E$ and $V$ glued together along this common boundary component.

\begin{lemma}\label{H1 isom}
For each $1 \le i \le n$ the map induced by the inclusion of $M_i \subset \bdry (W)$ into $W$ is an isomorphism on first homology.
\end{lemma}
\begin{proof}
Let $\mu_i$ denote the curve in $M_i$ given by the meridian.  By a Mayer-Vietoris argument 
\begin{equation}H_1(V) =  \langle \mu_1, \dots, \mu_n | \mu_1=\mu_2 = \dots =\mu_n\rangle \cong \Z\end{equation} 
is generated by the meridian of any component of $\bdry_- V$.  The meridian of $K$ in $\bdry_+ V$ is isotopic in $V$ to the meridian of any one of the components of $\bdry_- V$, so the inclusion of $\bdry_+ V$ into $V$ induces an isomorphism on $H_1$.  Additionally the inclusion of $\bdry_+ V = \bdry E$ into $E$ induces an isomorphism on $H_1$.  Combining these facts with another Mayer-Vietoris argument, one sees that $H_1(W) = \Z$ is generated by the meridian of any one of the boundary components of $W$, which completes the proof.
\end{proof}

\begin{lemma}\label{second homology by bdry}
The second homology of $W$ is carried by the boundary of $W$ so $\sigma(W)=0$.
\end{lemma}
\begin{proof}
By a Mayer-Vietoris argument using the decomposition of the 4-ball as the union of $E$ with a the neighborhood of a slice disk for $K$, $H_2(E) 
= 0$.  $V$ is homotopy equivalent to the union of its boundary together with 1-cells between different components and 2-cells glued to curves which are linearly independent in first homology so $V$ has second homology carried by its boundary.  As previously noted, the inclusion of any of the boundary components of $V$ into $V$ induces an isomorphism on $H_1$.  The Mayer-Vietoris long exact sequence associated to $W=V \cup E$ is
\begin{center}
$H_2(V)\oplus H_2(E) \overset{i_*}{\rightarrow}H_2(W)\overset{\bdry_*}{\rightarrow} H_1(M(K)) \overset {j_*\oplus k_*}{\longrightarrow} H_1(V)\oplus H_1(E)$
\end{center}
In this sequence, $j_*$ is an isomorphism so $j_*\oplus k_*$ is a monomorphism.  Thus, $\bdry_* = 0$ and $i_*$ is an epimorphism.  Since $H_2(E)=0$, $i_*$ is an epimorphism from $H_2(V)$ to $H_2(W)$.  Hence, $H_2(W)$ is carried by $H_2(V)$ which in turn is carried by $H_2(\bdry_- V) = H_2(\bdry W)$.
\end{proof}

\begin{lemma}\label{twisted homology}
Let $\phi:\pi_1(W)\to \Gamma$ be any PTFA coefficient system on $W$ with $\phi(\mu_i)\neq 0$ where $\mu_i$ is the meridian of any one of the boundary components of $W$.  Let $\K$ be the classical field of fractions of the Ore domian $\Q[\Gamma]$.  

Then $H_2(W;\K)=0$ so $\sigma^{(2)}(W;\phi)=0$.
\end{lemma}
\begin{proof}

The submanifold $V$ decomposes as the union of $\bdry_- V$ together with $n-1$ copies of $S^1\times B^2\times [0,1]$.  The intersection of these two sets is given by $2n-2$ copies of $S^1\times B^2$.  

The Mayer-Vietoris sequence corresponding to this decomposition with coefficients in $\K$ gives the exact sequence
\begin{equation}
\begin{array}{c}
H_p(\bdry_- V;\K)\oplus \iterate{\oplus}{i=1}{n-1}H_p(S^1\times B^2\times [0,1];\K)\to H_p(V;\K)\\\to \iterate{\oplus}{i=1}{2n-2}H_{p-1}(S^1\times B^2;\K).
\end{array}
\end{equation}

By \cite[corollary 3.12]{C}, for each $p$ $H_p(\bdry_- V;\K)$, $H_{p}(S^1\times B^2\times [0,1];\K)$ and $H_{p-1}(S^1\times B^2;\K)$ all vanish, so $H_p(V;\K)=0$.  

By \cite[2.10 b]{whitneytowers}, since there is an integral homology isomorphism from $S^1$ to $E$, $H_*(E;\K) = H_*(S^1;\K) = 0$.  Consider the Mayer-Vietoris exact sequence of the decomposition $W = V\cup E$.  Note that $V\cap E = M(K)$. 
$$H_p(E;\K)\oplus H_p(V;\K)\to H_p(W;\K)\to H_{p-1}(M(K);\K).$$
We have seen that $H_p(E;\K) = H_p(V;\K) = H_{p-1}(M(K);\K) = 0$, so $H_p(W;\K)=0$ for all $p$.  Taking $p=2$ completes the proof.

\end{proof}

For the reminder of this section we denote by $t$ the generator of the first homology of $W$.  Since the maps on integral first homology induced by the inclusion of every one of $V, E, M_i, \bdry_+V$ into $W$ is an isomorphism, the first homology of any one of these with coefficients in $\Q[t^{\pm1}]$ or $R_p$ is isomorphic to its Alexander module or localized Alexander module respectively.

\begin{lemma}\label{kernel is isotropic}
The submodule $P$ of $H_1\left(\bdry W; R_p \right)$ given by the kernel of the map induced by the inclusion of $\bdry W$ into $W$  is isotropic.  For any component $M_i$ of $\bdry W$ the submodule $Q$ of $H_1\left(M_i; R_p\right)$ given by the kernel of the map induced by inclusion into $W$ is isotropic.
\end{lemma}

\begin{proof}

  $R_p$ embeds in the field $\Q(t)$.  By Lemma~\ref{twisted homology} and the universal coefficient theorem with field coefficients \cite[Corollary 2.31]{DaKi} \begin{equation}H^2(W;\Q(t))=H_1(W;\Q(t))=0,\end{equation} so the Bockstein homomorphism \begin{equation}B:H^1(W; \Q(t)/R_p) \rightarrow H^2(W;R_p)\end{equation} is an isomorphism.  

Consider the following commutative diagram, where $P.D.$ denotes the Poincar\'e duality isomorphism and $\kappa$ denotes the Kronecker map.  The composition of the vertical maps on the right gives the Blanchfield form. 
\begin {equation}
\begin {diagram}
\node{H_2(W,\bdry W;R_p)}\arrow{s,r}{P.D.}\arrow{e,t}{\bdry_*}\node{H_1(\bdry W;R_p)}\arrow{s,r}{P.D.}\\
\node{H^2(W;R_p)}\arrow{s,r}{B^{-1}}\arrow{e,t}{i^*}\node{H^2(\bdry W;R_p)}\arrow{s,r}{B^{-1}}\\
\node{H^1(W,\Q(t)/R_p)}\arrow{s,r}{\kappa}\arrow{e,t}{i^*}\node{H^1(\bdry W;\Q(t)/R_p)}\arrow{s,r}{\kappa}\\
\node{\hom_{R_p}\left(H_1(W;R_p),\Q(t)/R_p\right)}\arrow{e,t}{\left(i_*\right)^{\dual}}\node{\hom_{R_p}\left(H_1(\bdry W;R_p), \Q(t)/R_p\right)}
\end{diagram}
\end{equation}

If $x$ and $y$ are elemets of $P$, then by the exactness of 
\begin{equation}
H_2(W,\bdry W;R_p)\overset{\bdry_*}{\to} H_1(\bdry W;R_p)\overset{i_*}{\to} H_1(W;R_p) 
\end{equation}
there are elements $X$ and $Y$ of $H_2(W, \bdry W; R_p)$ with $\bdry_* X = x$ and $\bdry_* Y = y$.  Thus, 
\begin{eqnarray*}
Bl(x,y) &=& ((\kappa \circ B^{-1} \circ (P.D.))(x))(y) \\&=& ((\kappa \circ B^{-1} \circ (P.D.)\circ \bdry_*)(X))(\bdry_*Y)
\end{eqnarray*}
using the commutative diagram above then
\begin{eqnarray*}
Bl(x,y) &=&((i_*^{\dual} \circ \kappa \circ B^{-1} \circ (P.D.))(X))(\bdry_*Y)\\
& = &(\kappa \circ B^{-1} \circ (P.D.))(X))(i_* \circ \bdry_* Y)
\end{eqnarray*}
and this is zero since $i_* \circ \bdry_* = 0$.  Thus, for any $x,y \in P$, $Bl(x,y)=0$ and $P\subseteq P^{\perp}$


Finally, $Q\subseteq P$ is contained in an isotropic submodule and so is isotropic.
\end{proof}


At this point we can begin the final stages of the proof of Theorem~\ref{big theorem 1}.  The regular signature of $W$ is zero since its integral second homology is carried by its boundary (by Lemma~\ref{second homology by bdry}).  The $L^2$ signature is zero by the inequality~\eqref{inequality} since $H_2(W; \K)=0$ (by Lemma~\ref{twisted homology}).

By Lemma~\ref{H1 isom} the homomorphism induced by inclusion of each boundary component $M_i$ into $W$ on first homology is an isomorphism.  By Lemma~\ref{kernel is isotropic} and the assumption that there are no nontrivial isotropic submodules, it must be that the induced map on localized Alexander modules is an injection.  Lemma~\ref{repeated use lemma} now asserts that the inclusion induced map from $\dfrac{\pi_1(M_i)}{\pi_1(M_i)^{(2)}_p}$ to $\dfrac{\pi_1(W)}{\pi_1(W)^{(2)}_p}$ is injective and by definition~\ref{rho} \begin{equation}\displaystyle \sum_{i=1}^n \rho^1_p(K_i) = \sigma^{(2)}(W;\phi^p)-\sigma(W) = 0.\end{equation}  This concludes the proof of Theorem~\ref{big theorem 1}.
\end{proof}
\begin{remark}
It is not necessary that $E$ be a slice disk complement.  The only properties of $E$ that we use are the following: 
\begin{enumerate}
\item $\bdry E = M(K)$ and the map induced by inclusion on first homology is an isomorphism.
\item The kernel of the map induced by inclusion on first homology with coefficients in $R_p$ is isotropic.
\item The signature defect of $E$ corresponding to the quotient map \begin{equation*}\pi_1(E)\to\pquotient{E}{p(t)}\end{equation*} is zero.
\end{enumerate}
These conditions are satisfied when $E$ is a (1.5)-solution for $K$ \cite[Theorems 4.2 and 4.4]{whitneytowers}.  Thus, the word slice can be replaced with $(1.5)$-solvable and the concept of linear dependence in $\C$ can be replaced with linear dependence in $\C/\mathcal{F}_{(1.5)}$ in Theorem~\ref{big theorem 1} and Corollaries~\ref{torsion corollary 1} and \ref{big corollary 1}.
\end{remark}


\section{Relating $\rho^1_p$ with $\rho^0$}\label{computational tools}

Let $K$ be a knot which is of finite order $n >1$ in the algebraic concordance group.  Let $\Sigma$ be a genus $g$ Seifert surface for $\underset{n}{\#}K$.  Let $L$ be a link of $g$ components sitting on $\Sigma$ which represents a metabolizer of the Seifert form.  This section establishes a relationship between $\rho^1_p(K)$ and $\rho^0(L)$.  {Notice that since $L$ is a metabolizer, it must have zero pairwise linking numbers so that $\rho^0(L)$ is defined.}  The resulting theorem (Theorem~\ref{premain}) is used in Section~\ref{twist} to get information about $\rho^1$.

We need the following piece of notation:
\begin{definition}\label{band}
For a knot $K$ with Seifert surface $\Sigma$ and a link $L$ sitting on $\Sigma$, let $\gamma$ be a component of $L$.  A curve $m$ which does not intersect $\Sigma$ is called a \textbf{meridian for the band on which $\gamma$ sits} if $m$ bounds a disk in $S^3$ which intersects $\Sigma$ in an arc which crosses $\gamma$ in a single point and does not intersect any other component of $L$.
\end{definition}


\begin{construction}\label{W for computation}
Let $V$ be as in the proof of Theorem~\ref{big theorem 1}, so $\bdry_+ V$ is given by $M(J)$, where $ J = \underset{n}{\#} K$.  Thinking of $L$ as sitting in $\bdry_+ V$, let E be given by adjoining to $V$ a two handle along the zero framing of each component of $L$.  Let $\bdry_- E$ be the disjoint union of the $M(K)$ boundary components of $E$.  Let $\bdry_+ E$ be the boundary component of $E$ given by zero surgery on $L$ together with $J$, that is $\bdry_+ E = M(L \cup J)$.

Sliding $J$ over the handles attached to $L$, one sees that $J$ bounds a disk in $M(L)$ and $\bdry_+E\cong M(L)\#S^2\times S^1$.  Adjoin to $V\cup E$ a three handle along the nonseperating $S^2$ in $\bdry_+ E$ and call the resulting 4-manifold $W$.  We denote by $\bdry_+W $, the $M(L)$ component of $\bdry W$ and by $\bdry_-W=\bdry_-V$, the disjoint union of the $M(K)$ components of $\bdry W$.  The submanifold $W - V$  is identical to the manifold constructed in \cite[8.1]{derivatives}, where it is called $E$ and is described in more detail.  
\end{construction}

Let $p\in \Q[t^{\pm1}]$ be a polynomial.

We wish to use $W$ to make a claim involving the $\rho^1_p$-invariant of $K$ and the $\rho^0$-invariant of $L$.  We begin with an overview of the strategy we use to do so.  Let $\phi:\pi_1(W)\to\pquotient{W}{p}$ be the projection map.  Lemmas~\ref{H2 by bdry} and \ref{H1 injection with coeff} together with Lemma~\ref{repeated use lemma} will give that \begin{equation}n\rho^1_p(K)-\rho^0(L)=\sigma^{(2)}(W;\phi)-\sigma(W).\end{equation}  Lemma~\ref{bound} will give bounds on $\sigma^{(2)}(W;\phi)-\sigma(W)$.

\begin{lemma}\label{H2 by bdry}
\begin{enumerate}
\item
For each $M(K)$-component of $\bdry_- W$ the map induced by inclusion from $H_1(M(K))$ to $H_1(W)$ is an isomorphism.
\item
The kernel of the map induced by inclusion from $A_0^p(K)$ to $A_0^p(W)$ is isotropic.
\end{enumerate}
\end{lemma}
\begin{proof}

The inclusion from $M(K)$ into $V$ induces a first homology isomorphism, as was observed in the proof of Lemma~\ref{H1 isom}.  $W$ is obtained by adding 2-cells to null-homologous curves in $V$ and then adding a 3-cell.  Neither of these operations changes first homology so the inclusion from $M(K)$ into $W$ induces a first homology isomorphism, which proves (1).

If $x,y\in A_0^p(K)$ are in the kernel of the inclusion induced map, then $$(x,0,0,\dots,0), (y,0,0,\dots,0) \in \underset{n}{\oplus} A^p_0(K)= A_0^p\left(\underset{n}{\#} K\right)$$ are in the submodule $P$ generated by the homology classes of the lifts of the components of $L$.  
By Lemma~\ref{metabolic to Lagrangian}, which is stated and proved at the end of this section, $P$ is isotropic with respect to the linking form on $A_0^p\left(\iterate{\#}{n}{} K\right)$, which is precisely the n-fold direct sum of the linking form of $A_0^p(K)$ with itself. 

Thus, $Bl_K(x,y)=Bl_{\# K}((x,0,\dots,0),(y,0,\dots,0)) = 0$ and the kernel of the inclusion induced map is isotropic with respect to Blanchfield linking.

\end{proof}

\begin{theorem}\label{H1 injection with coeff}
\vspace{.1in}
\noindent\begin{enumerate}
\item If $K$ is a $p$-anisotropic knot, then the map from $\pquotient{M(K)}{p}$ to $\pquotient{W}{p}$ induced by the inclusion of any one of the $M(K)$-components of $\bdry_- W$ is injective.

\item Let $m_i$ be the meridian about the band of $\Sigma$ on which $L_i$ sits and $P$ be the submodule of $A_0^p(J)$ generated by $L$. If $m_1,\dots, m_n$ are $\Z$-linearly as elements of $A_0^p(J)/P$, then the inclusion of $M(L) = \bdry_+ W$ into $W$ induces an injective map from $\nquotient{M(L)}{1}$ to $\pquotient{W}{p}$
\end{enumerate}
\end{theorem}
\begin{proof}

By Lemma~\ref{H2 by bdry} part 2 the kernel of this map on $A_0^p$ is isotropic, but by assumption the only such submodule of $A_0^p(K)$ is zero so the inclusion induced map \begin{equation}A_0^p(K)\hookrightarrow A_0^p(W)\end{equation} is a monomorphism.  From this together with Lemma~\ref{H2 by bdry} part 1 and Lemma~\ref{repeated use lemma} it follows that the induced map \begin{equation}\pquotient{M(K)}{p} \to \pquotient{W}{p}\end{equation} is a monomorphism, completing the proof of (1).

By \cite[prop 8.1 (5)]{derivatives} The inclusion induced map from $H_1(M(L))$ to $H_1(W)$ is trivial.  Thus, the inclusion induced map sends $\pi_1(M(L))$ to $\pi_1(W)^{(1)}$.  Consider the composition 
\begin{equation}\label{composition}H_1(M(L))\to\dfrac{\pi_1(W)^{(1)}}{\pi_1(W)^{(2)}}\to\onepquotient{W}{p}\to  A_0^p(W) = A_0^p(J)/P\end{equation}
By \cite[prop 8.1 (3)]{derivatives} the generators of the left hand side of (\ref{composition}) (meridians of the components of $L$) are isotopic in $W$ to the meridians of the bands on which the components of $L$ sit.  By assumption, the latter form a $\Z$-linearly independent set on the right hand side of (\ref{composition}), so this composition is injective.  In particular this means that the composition of the two leftmost terms in (\ref{composition}) is injective, so $H_1(M(L))\to\onepquotient{W}{p}\subseteq \pquotient{W}{p}$ is a monomorphism which completes the proof.

\end{proof}

Theorem~\ref{H1 injection with coeff} implies 
\begin{equation}\label{use W}
\sigma^{(2)}(W,\phi) - \sigma(W) = n \rho^1_p(K)-\rho^0(L).
\end{equation}  
Lemma~\ref{bound} bounds the signatures on the left hand side of this equation.  It is independent of any anisotropy assumption.  Before we can address this lemma we must provide a definition for the Alexander nullity of a link.

In \cite[Definition 7.3.1]{Ka3} the \textbf{Alexander nullity} of an $m$-component link $L$ is defined as \begin{equation*}\eta(L)=\rank_{\Q(\Z^n)}(H_1(E(L),x;\Q(\Z^n))-1\end{equation*} where $E(L)=S^3-L$ is the exterior of the link, $x$ is a point in $E(L)$ and the homology is twisted by the abelianization map.  A study of the long exact sequence of $(E(L),x)$ reveals that \begin{equation*}\eta(L)=\rank_{\Q(\Z^n)}(H_1(E(L);\Q(\Z^n)).\end{equation*}  In the case of a link with pairwise zero linking a Mayer Veitoris argument reveals that \begin{equation}\label{eta}\eta(L)=\rank_{\Q(\Z^n)}(H_1(M(L);\Q(\Z^n)).\end{equation}  The last of these interpretations is most convenient for our purposes.

\begin{lemma}\label{bound} For the 4-manifold $W$ given in Construction~\ref{W for computation} and the quotient map $\phi:\pi_1(W)\to\pquotient{W}{p}$.  

\vspace{.1in}
\noindent\begin{enumerate}
\item$\sigma(W)=0$ 
\item$\left|\sigma^{(2)}(W,\phi) \right| \le g-1-\eta(L)$ where $\eta(L)$ is the Alexander nullity of $L$
\end{enumerate}
\end{lemma}
\begin{proof}
To see the first claim notice that the homology of W which is not carried by the $M_i$ components of $\bdry_- W$ is generated by 2-handles attached to $M(K)$ along the zero framings of $g$ curves which have zero linking numbers with each other.  The intersection form is thus given by the $g\times g$ zero matrix, which proves (1).

Let $\K$ be the classical field of fractions of the Ore domain $\Q\left[\pquotient{W}{p}\right]$.  To see the second claim we show that \begin{equation*}\rank_\K \left(\dfrac{H_2\left(W;\K\right)}{i_*\left[H_2\left(\bdry W;\K\right)\right]}\right) =  g-1-\eta(L).\end{equation*}

The pair $(W,V)$ consists of g relative 2-handles and 1 relative 3-handle so $\chi(W)-\chi(V)=g-1$.  Since $V$ has the homotopy type of a (disconnected) closed 3-manifold together with a $n$ 1-cells and $n$ 2-cells, $\chi(V)=0$.  It must be that $\chi(W)=g-1$.  By \cite[Proposition 3.7]{C}, $H_0(W;\K)=0$.  By \cite[Proposition 3.10]{C}, $H_1(W;\K)=0$.  $W$ has the homotopy type of a 3-complex, so $H_4(W;\K)=0$.  Consider the long exact sequence of the pair $(W,\bdry W)$,
$$H_3(\bdry W;\K)\to H_3(W;\K)\to H_3(W,\bdry W;\K).$$
Employing Poincar\'e duality and the universal coefficient theorem over the skew field $\K$, $H_3(\bdry W;\K) = H_0(\bdry W;\K)=0$ and $H_3(W,\bdry W;\K)=H_1(W;\K)=0$.  Thus, $H_3(W;\K)=0$.

Since the alternating sum of the ranks of twisted homology gives the Euler characteristic, $\rank_\K \left(H_2(W;\K)\right) = \chi(W) = g-1$.

By Theorem~\ref{H1 injection with coeff} part 2, $$H_1(\bdry_+ W;\K) = H_1(M(L);\Q(\Z^n))\otimes \K =\K^{\eta(L)},$$
so since $H_1(\bdry_-(W);\K) = 0$, $H_1(\bdry W;\K) = \K^{\eta(L)}$.  By Poincar\'e duality, $H_2(\bdry W;\K)=H_1(\bdry W;\K)=\K^{\eta(L)}$.  

Since $H_3(W,\bdry W;\K)=0$, the exact sequence of the pair indicates that $i_*:H_2(\bdry W;\K)\to H_2(W;\K)$ is a monomorphism.

Thus, \begin{eqnarray*}\rank_\K \left(\dfrac{H_2\left(W;\K\right)}{i_*\left[H_2\left(\bdry W;\K\right)\right]}\right) &=& \rank_\K \left(H_2(W;\K)\right)-\rank_\K \left(H_2(\bdry W;\K)\right)\\ &=& (g-1)-\eta(L).\end{eqnarray*}

Finally, $\left|\sigma^{(2)}(W)\right| \le \rank_\K \left(\dfrac{H_2\left(W;\K\right)}{H_2\left(\bdry W;\K\right)}\right)$ by inequality~\eqref{inequality}, which completes the proof.

\end{proof}

Now by equation~\ref{use W}, $\sigma^{(2)}(W,\phi)-\sigma(W) = n\rho^1_p(K) - \rho^0(L)$.  Using the bound obtained in Lemma~\ref{bound} the theorem below is proven. 


\begin{theorem}\label{premain}
Let $K$ be a $p$-anisotropic knot of finite algebraic order $n>1$.  Let $\Sigma$ be a genus $g$ Seifert surface for $\underset{n}{\#}K$.  Let $L$ be a link of $g$ curves on $\Sigma$ representing a metabolizer for the Seifert form.   Let $P$ be the submodule of  $A^p_0\left({\#}K\right)$ generated by $L$.  Suppose that the meridians about the bands on which the components of $L$ sit form a $\Z$-linearly independent set in $A_0^p\left({\#}K\right)/P$.  Then
$$ \left| n\rho^1_p(K)-\rho^0(L)\right| \le g-1-\eta(L).$$
\end{theorem}

All of the results of this section could be rephrased to deal with the case that $K_1\dots K_n$ are (possibly distinct) $p$-anisotropic knots and  $J=\iterate{\#}{i=1}{n}K_i$ is algebraically slice with a metabolizer $L$.  Following the exact same argument one gets a stronger theorem. 

\begin{theorem}\label{postmain}
Suppose that $K_1\dots K_n$ are (not necessarily distinct) $p$-anisotropic knots and $K=\iterate{\#}{i=1}{n}K_i$ is algebraically slice.  Let $\Sigma$ be a genus $g$ Seifert surface for $K$.  Let $L$ be a link of $g$ curves on $\Sigma$ representing a metabolizer for the Seifert form.  Let $P$ be the submodule of $A_0^p(K)$ generated by $L$.  Suppose that the meridians about the bands on which the components of $L$ sit form a $\Z$-linearly independent set in $A_0^p(K)/P$, where $P$ is the submodule of $A_0^p(K)$ generated by $L$.  Then $$\displaystyle \left| \sum_{i=1}^{n} \rho_p^1(K_i)-\rho^0(L)\right| \le g-1-\eta(L).$$
\end{theorem}

\begin{example} 
We provide an application of Theorem~\ref{postmain}.  Suppose that $L$ is a 2 component link with pairwise linking number zero and such that $\left|\rho^0(L)\right|>1-\eta(L)$.  Such links arise as by-products of the analysis in the next section.  Let $i$ be a positive integer.  The link $L$ is a metabolizer for the algebraically slice knot $K_i$ depicted in Figure~\ref{fig:antiderivative}.  Notice that $K_i$ is given by the connected sum of $J_i(L)$ and $J_i$ both of which have prime Alexander polynomials and so are anisotropic.   The meridians of the bands on which $L$ sit are depicted in Figure~\ref{fig:antiderivative}.  For some choice of basis of the Alexander module of $K_i$, they represent 
\begin{equation}
m_1=\left[\begin{array}{c}
1\\0
\end{array}\right],
m_2=\left[\begin{array}{c}
i(t-1)\\0
\end{array}\right].
\end{equation}
While the components of $L$ represent
\begin{equation}
l_1=\left[\begin{array}{c}
ti\\-i
\end{array}\right],
l_2=\left[\begin{array}{c}
i^2(1-t)\\i^2(t-1)
\end{array}\right]
\end{equation}
in  
\begin{equation}
\displaystyle A_0(K_i) = \dfrac{\Q[t^{\pm1}]}{(i^2t^2+(1-2i^2)t+i^2)} \oplus \dfrac{\Q[t^{\pm1}]}{(i^2t^2+(1-2i^2)t+i^2)}.
\end{equation}  Notice that $m_1, m_2, l_1, l_2$ form a $\Z$-linearly independent set.  

\begin{figure}[h]
\setlength{\unitlength}{1pt}
\begin{picture}(200,90)
\put(00,-65){\includegraphics[width=1\textwidth]{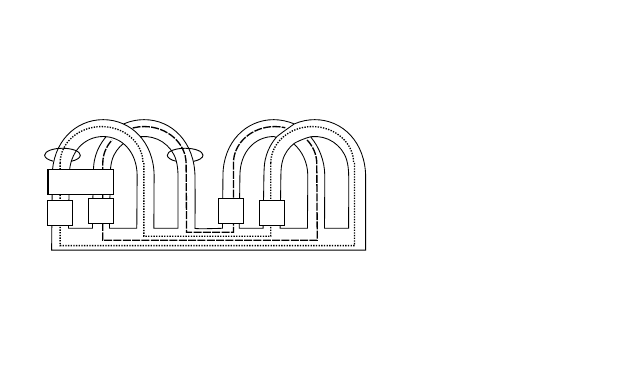}}
\put(60,0){$J_i(L)$}
\put(160,0){$J_i$}
\put(134,28){$i$}
\put(148,28){$-i$}
\put(32,28){$i$}
\put(51,28){$-i$}
\put(42,45){$L$}
\put(20,73){$m_1$}
\put(112,70){$m_2$}

\end{picture}
\caption{$J_i(L)\#J_i$ has the link $L$ as a derivative.  }\label{fig:antiderivative}
\end{figure}
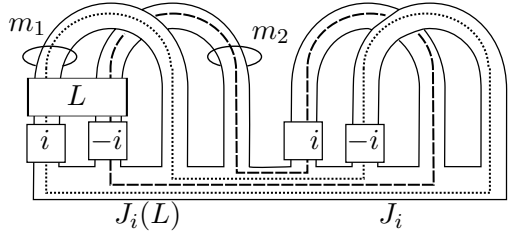

Theorem~\ref{postmain} applies to give that $\rho^1(J_i(L))+\rho^1(J_i)$ is nonzero.  Since $J_i$ is of topological order 2, Theorem~\ref{big theorem 1} gives us that $\rho^1(J_i)=0$.  It must be that $\rho^1(J_i(L))$ is not zero.  For $i\neq k$ positive integers $J_i(L)$ and $J_k(L)$ have distinct prime Alexander polynomials so by Corolary~\ref{big corollary 1} the knots $J_i(L)$ $i>0$ are linearly independent in $\C$.
\end{example}

\subsection{Metbolizers of the Seifert surface and isotropy in the Alexander module}

In this sub-section we state and prove a fact used in the proof of Lemma~\ref{H2 by bdry}.  The proof relies on the formula in \cite[section 8]{BlDuality} for the unlocalized Blanchfield form in terms of the Seifert matrix.  


\begin{lemma}\label{metabolic to Lagrangian}
Suppose that $\Sigma$ is a genus $g$ Seifert surface for a knot $K$ and $L = L_1\dots L_g$ is a $g$-component link sitting on $\Sigma$ which spans a rank $g$ direct summand of $H_1(\Sigma)$ on which the Seifert form vanishes.  
The submodule of $A_0^p(K)$ generated by the components of $L$ is isotropic.
\end{lemma}

\begin{proof}


We begin by proving that the submodule of the unlocalized Alexander module, $A_0(K)$, generated by $L$ is isotropic.  We call this submodule $P$.

Let the set $\{L_1\dots L_g\}$ be extended to  $\{L_1\dots L_g, D_1, \dots D_g\}$, a symplectic basis for $H_1(\Sigma)$.  Let $\mu_1, \dots \mu_g, \nu_1, \dots, \nu_g$ be the dual basis for $H_1(S^3 - \Sigma)$ given by meridians about the bands on which $L_i$ and $D_i$ sit.  The homology classes of the lifts of $\mu_1, \dots \mu_g, \nu_1, \dots, \nu_g$ to the infinite cyclic cover of $M(K)$ form a generating set for $A_0(K)$ as a $\Q$-vector space.  The map from $H_1(\Sigma)$ to the Alexander module induced by lifting $\Sigma$ to the cyclic cover of $M(K)$ is given with respect to these generating sets by the Seifert matrix $V$.  The Blanchfield form with respect to the generating set given by the lifts of $\mu_1, \dots \mu_g, \nu_1, \dots, \nu_g$ is given  by \begin{equation}\label{Blanchfield as seifert}Bl(\vec{r}, \vec{s})=(1-t)(\vec{s})^T(V-tV^T)^{-1}(\vec r)\end{equation} (see \cite[section 8]{BlDuality}).

Since $\{L_1\dots L_g\}$ is a metabolizer for the Seifert form, $V$ is given by a matrix of the form $\left[\begin{array}{cc}0&A\\B&C\end{array}\right]$, with respect to the basis $\{L_1\dots L_g, D_1, \dots D_g\}$ for $H_1(\Sigma)$ ($A,B,C$ are $g\times g$ matrices).  Thus, $(V-tV^T)^{-1}$ is given by a matrix with entries in $\Q[t^{\pm1}]$ of the form $\left[\begin{array}{cc}D&E\\F&0\end{array}\right]$, where $D,E,F$ are $g\times g$ matrices with polynomial entries.  

Consider any $\vec r=V\left[\begin{array}{c}\vec a\\0\end{array}\right]$, $\vec s=V\left[\begin{array}{c}\vec b\\0\end{array}\right]$ in $P$ ($a$ and $b$ are $g$-dimensional column vectors while $0$ denotes the $g$-dimensional zero vector).  Plugging these values into (\ref{Blanchfield as seifert}) we see
\begin{eqnarray*}
Bl(\vec r, \vec s)=(1-t)\left[\begin{array}{c}\vec b\\0\end{array}\right]^T\left[\begin{array}{cc}0&B^T\\A^T&C^T\end{array}\right]\left[\begin{array}{cc}D&E\\F&0\end{array}\right]\left[\begin{array}{cc}0&A\\B&C\end{array}\right]\left[\begin{array}{c}\vec a\\0\end{array}\right].
\end{eqnarray*}
This is zero by direct computation.  Since $P$ is the submodule generated by the lifts of $L_1\dots L_g$, this shows that $P$ is isotropic.

Now consider any $\alpha\otimes \frac{e}{f}$ and $\beta\otimes \frac{g}{h}$ in $P\otimes R_p$, the submodule of $A_0^p(K)\cong A_0(K)\otimes R_p$ generated by $L$.  By the sesquilinearity of the Blanchfield form, 
$$Bl(\alpha\otimes \frac{e}{f},\beta\otimes \frac{g}{h})=Bl(\alpha,\beta)\frac{e\overline{g}}{f\overline{h}},$$
where $Bl(\alpha,\beta)$ is given by the unlocalized Blanchfield form.  Since $\alpha$ and $\beta$ sit in $P\subseteq A_0$, which we have just shown to be isotropic, $Bl(\alpha,\beta)$ is zero so that $P\otimes R_p$ is isotropic.
\end{proof}
\section{An infinite family of twist knots of algebraic order $2$ with distinct prime Alexander polynomials and nonzero $\rho^1$-invariant.} \label{twist}


Consider $T_n$ the n-twist knot depicted in Figure~\ref{fig:twist}.  The goal of this section is the proof of Theorem~\ref{rho twist}.

\begin{theorem}\label{rho twist}
Let $n(x) = -x^2-x-1$.  If $x\ge2$, then $\rho^1\left(T_{n(x)}\right)<0$.
\end{theorem}

From this theorem we get the immediate corollary.

\begin{corollary}\label{twist theorem}
The twist knots $T_{-7}, T_{-13}, T_{-21}, \dots, T_{-x^2-x-1}, \dots$ form a linearly independent set $\C$.
\end{corollary}
\begin{proof}[Proof of Corollary~\ref{twist theorem}]
These twist knots have nonzero $\rho^1$-invariants and distinct prime Alexander polynomials $\Delta_{T_{n}}(t)=nt^2+(1-2n)t+n.$ By Corollary~\ref{big corollary 1} these knots form a linearly independent set.  
%

\end{proof}

We move on to the proof of Theorem~\ref{rho twist}, which occupies us for the remainder of this paper.

For each $x$, the knot $T_{n(x)}$ is algebraically of order two (see \cite[Corollary 23]{Le10}).  The knot $T_{n(x)}\#T_{n(x)}$ has the following as its Seifert matrix taken with respect to the obvious basis for the first homology of the Seifert surface depicted in Figure~\ref{fig:metabaslink}.
\begin{equation}\label{seif mat}
\left[
\begin{array}{cccc}
n(x)&1&0&0\\
0&1&0&0\\
0&0&n(x)&1\\
0&0&0&1
\end{array}
\right].
\end{equation}
This matrix has a metabolizer generated by 
\begin{equation}
v_1 = \left[
\begin{array}{c}
1\\
x\\
0\\
1
\end{array}
\right], 
v_2=\left[
\begin{array}{c}
0\\
1\\
1\\
-x-1
\end{array}
\right].
\end{equation}

\begin{figure}[b]
\setlength{\unitlength}{1pt}
\begin{picture}(200,90)
\put(0,-30){\includegraphics[width=.55\textwidth]{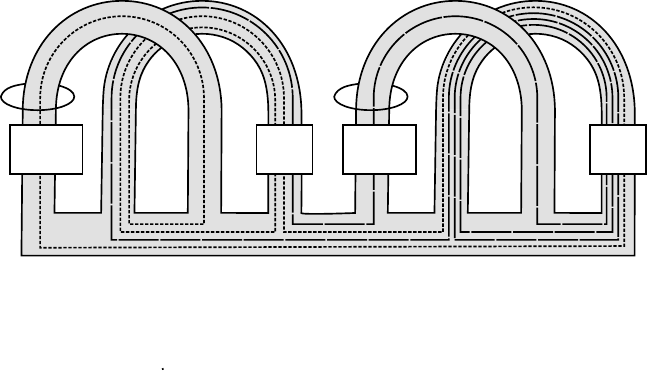}}
\put(80,34){$+1$}
\put(182,34){$+1$}
\put(4,35){$n(x)$}
\put(106,34){$n(x)$}
\put(-5,62){$m_1$}
\put(96,62){$m_2$}
\put(50,90){$x+1$ strands}
\put(160,90){$x+2$ strands}
\end{picture}
\caption{$L_x$ as a link in $S^3$ sitting on a Seifert surface for $T_{n(x)}\#T_{n(x)}$.  (In this picture $x=2$.)}\label{fig:metabaslink}
\end{figure}
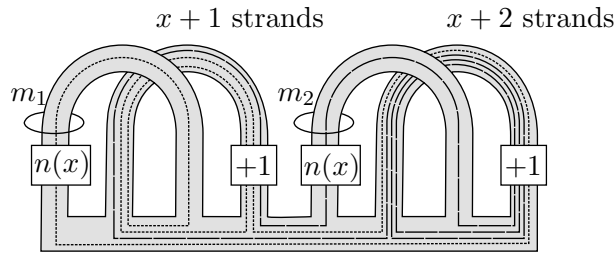

As homology classes, $v_1$ and $v_2$  are represented by the link $L_x$ also depicted in Figure~\ref{fig:metabaslink}.  Meridians for the bands on which the components of $L_x$ sit (also depicted in Figure~\ref{fig:metabaslink}) represent generators for the Alexander module: 
\begin{equation}
m_1=\left[\begin{array}{c}
1\\0
\end{array}\right],
m_2=\left[\begin{array}{c}
0\\1
\end{array}\right]
\end{equation}
while the components of the link are given by:
\begin{equation}
l_1=\left[\begin{array}{c}
t(n(x)+xn(x)+x)+1\\t(n(x)+1)-1
\end{array}\right],
l_2=\left[\begin{array}{c}
t(n(x)+1)-1\\t(-xn(x)-n(x)-1)+x+1
\end{array}\right].
\end{equation}
In \begin{equation}A_0(T_{n(x)}\#T_{n(x)}) = \underset{2}{\oplus}\left(\dfrac{\Q[t^{\pm1}]}{(n(x)t^2+(1-2n(x))t+n(x))}\right)\end{equation}  the elements $m_1, m_2, l_1, l_2$ form a  $\Z$-linearly independent set.


Theorem~\ref{premain} provides us with a strategy to prove Theorem~\ref{rho twist}.  Specifically, if $\rho^0(L_x) < -1$ then $\rho^1(T_n)<0$.  


\begin{theorem}\label{twist2}
When $x$ is an integer greater than $1$, $\rho^0(L_x)<-1$.
\end{theorem} 

\begin{proof}

An important workhorse in the proof of this theorem is Lemma~\ref{surgery}.  It gives bounds on how surgery along a nullhomologous curve changes $\rho^0$.  We use this theorem to reduce the computation of $~\rho^0(L_x)$ for every $x$ to a single computation for a simpler link.  The proof of the lemma is easy and we leave it to the end.

\begin{lemma}\label{surgery}
If a link $L'$ is given by performing $+1$ surgery on $L$ along a nullhomologous curve in the complement of $L$ then $\rho^0(L')\le\rho^0(L)$.
\end{lemma}


As is depicted in Figure~\ref{fig:L3fromL2}, the link $L_x$ can be realized as $+1$ surgery along nullhomologous curves on $L_{x-1}$.  Proposition~\ref{surgery} then implies that for $x>2$ \begin{equation}\label{Lx < L2}\rho^0(L_x)\le\rho^0(L_{x-1})\le\dots \le \rho^0(L_2).\end{equation}  

\begin{figure}[h]
\setlength{\unitlength}{1pt}
\begin{picture}(210,100)
\put(00,0){\includegraphics[width=.6\textwidth]{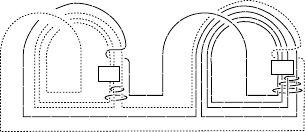}}
\put(70,38){{$+1$}}
\put(194,43){$+1$}
\end{picture}
\caption{If one performs $+1$ surgery on the depicted nullhomologous curves, then the link depicted is $L_{x+1}$.  If the surgery curves are erased, the link is $L_x$ ($x=2$).}\label{fig:L3fromL2}
\end{figure}

The link $L_2$ is realized in Figure~\ref{fig:L2} as $+1$ surgery along the unlink along four nullhomologous curves.  Consider the link $L'$ depicted in Figure~\ref{fig:L'} obtained by performing only two of these four surgeries.  By Lemma~\ref{surgery},  \begin{equation}\label{L2 < L'}\rho^0(L_2)\le\rho^0(L').\end{equation}

\begin{figure}[ht]
\setlength{\unitlength}{1pt}
\begin{picture}(210,100)
\put(0,0){\includegraphics[width=.6\textwidth]{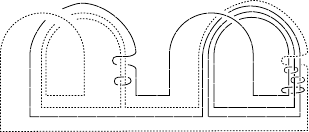}}
\end{picture}
\caption{The link $L_2$ as the result of $+1$ surgery on the unlink along nullhomologous curves.}\label{fig:L2}
%
\setlength{\unitlength}{1pt}
\begin{picture}(210,100)
\put(00,0){\includegraphics[width=.6\textwidth]{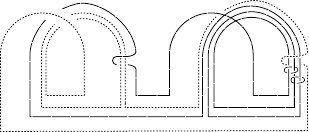}}
\put(220,35){$\gamma_1$}
\put(185,46){$\gamma_2$}
\end{picture}
\caption{the link $L'$ as the result of $+1$ surgery on the unlink along fewer curves than $L_2$.}\label{fig:L'}
\end{figure}

\begin{figure}[h]
\setlength{\unitlength}{1pt}
\begin{picture}(210,100)
\put(0,0){\includegraphics[width=.6\textwidth]{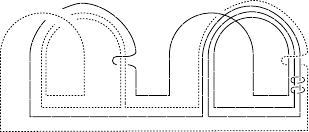}}
\put(223,42){$\gamma_1$}
\put(223,25){$\gamma_2$}
\end{picture}
\caption{the link $L'$ after an isotopy}\label{fig:alsoL'}
\end{figure}


The diagram for $L'$ in Figure~\ref{fig:L'} is of $+1$  surgery on nullhomologous curves $\gamma_1$, $\gamma_2$ on the unlink.  Using this fact we build a 4-manifold with $M(L')$ as its boundary.  Start with $V$, the boundary connected sum of two copies of $S^1\times B^3$.  The boundary of $V$ is zero surgery on the two component unlink.  Thinking of $\gamma_1$ and $\gamma_2$ as sitting on $\bdry V$, attach two-handles to their $+1$ framings and call this manifold be $W$.  the inclusion induced map from $H_1(M(L'))$ to $H_1(W)$ is a monomorphism so \begin{equation}\label{W computes L'}\rho^0(L')=\sigma^{(2)}(W,\phi)-\sigma(W),\end{equation} where $\phi$ is the abelianization map on $H_1(W)$.


The untwisted intersection matrix of $W$ is given by the $2\times2$ identity matrix so $\sigma(W)=2$.  

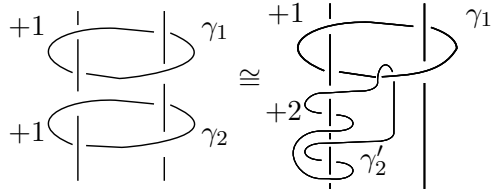
\begin{figure}[h]
\setlength{\unitlength}{1pt}
\begin{picture}(180,70)
\put(0,-20){\includegraphics[width=.3\textwidth]{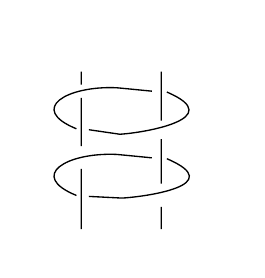}}
\put(100,-30){\includegraphics[width=.3\textwidth]{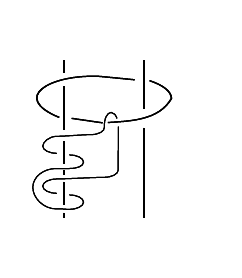}}
\put(94,35){$\cong$}
\put(7,15){$+1$}
\put(7,55){$+1$}
\put(80,55){$\gamma_1$}
\put(80,15){$\gamma_2$}

\put(105,60){$+1$}
\put(104,23){$+2$}
\put(180,60){$\gamma_1$}
\put(140,5){$\gamma_2'$}
\end{picture}
\caption{Right: A closer view of the surgery curves in the diagram for $L'$ in Figure~\ref{fig:alsoL'}. Left: A diagram for $L'$ gotten by performing a handle slide.}\label{fig:alsoL'zoom}
\end{figure}

Consider the result of performing the handle slide depicted in Figure~\ref{fig:alsoL'zoom} on the diagram in Figure~\ref{fig:alsoL'} (an isotopy of the diagram in Figure~\ref{fig:L'}).  The resulting diagram for $W$ is of a 4-manifold gotten from $V$ by first gluing a 2-handle to the $+2$ framing of a null-homotopic curve $(\gamma_2')$ and then another to a curve which is non-torsion in $H_1(V;\Q[\Z^2])$ $(\gamma_1)$. The second of these additions affects neither twisted second homology nor the twisted intersection form.  Let $W'$ be the 4-manifold gotten by gluing a two handle to $V$ along the $+2$ framing of $\gamma'_2$.  By the above, \begin{equation}\label{W equals W'}\sigma^{(2)}(W,\phi)=\sigma^{(2)}(W',\phi).\end{equation}  A Kirby diagram for $W'$ is given in Figure~\ref{fig:W'}.  A convenient istopy of the diagram is given in Figure~\ref{fig:W'isotope}.

\begin{figure}[h]
\setlength{\unitlength}{1pt}
\begin{picture}(160,70)
\put(0,-50){\includegraphics[width=.6\textwidth]{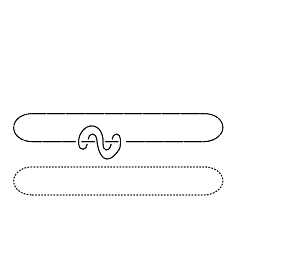}}
\put(45,55){$+2$}
\put(75,56){$\gamma_2'$}
\put(55,6){$\bullet$}
\put(55,63){$\bullet$}
\end{picture}
\caption{A Kirby diagram for $W'$.}\label{fig:W'}
%
\setlength{\unitlength}{1pt}
\begin{picture}(150,90)
\put(00,7){\includegraphics[width=.4\textwidth]{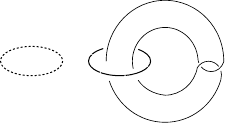}}
\put(133,80){$+2$}
\put(127,55){$\gamma_2'$}
\put(17,35){$\bullet$}
\put(75,35.5){$\bullet$}
\end{picture}
\caption{The diagram in Figure~\ref{fig:W'} after an isotopy.  It is convenient that the $+2$ pushoff of $\gamma_2'$ is the blackboard pushoff in this diagram.}\label{fig:W'isotope}
\end{figure}

%

Notice that while $\gamma_2'$ does not bound an embedded disk in $\bdry V$, its lift to the abelian cover of $\bdry V$ does.  This disk $D$ is depicted in Figure~\ref{fig:isect'n}.  The twisted second homology of $W'$ is $\Q[t^{\pm1},s^{\pm1}]$ generated by the 2-sphere $S$ given by the core of the 2-handle glued to $\gamma_2'$ together with $D$.  The twisted intersection matrix of $W$ is given by the equivariant self intersection of $S$.  The sphere $S$ has a pushoff which crosses the $S$ exactly where $D$ intersects the lifts of the $+2$ pushoff of $\gamma_2'$.  Thus, the self intersection matrix of $S$ is given by counting the number of times (with coefficients) that $D$ intersects $\gamma_2'$.  These intersection points are depicted in Figure~\ref{fig:isect'n}.

\begin{figure}[t]
\setlength{\unitlength}{1pt}
\begin{picture}(340,160)
\put(0,45){\includegraphics[width=.9\textwidth]{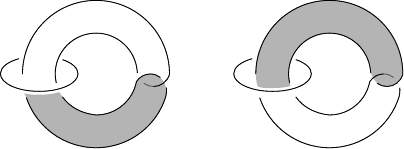}}
\put(115,50){$\gamma'_2$}
\put(130,145){$t \gamma'_2$ }
\put(310,50){$t^{-1}\gamma'_2$}
\put(320,155){$\gamma'_2$}
\put(12,27){The portion of $D$ sitting in}
\put(12,14){sheet 1 of $\widetilde{V}$ contributes $+t$ }
\put(12,1){to the self intersection of $S$.}
\put(202,27){The portion of $D$ sitting in}
\put(202,14){sheet $t$ of $\widetilde{V}$ contributes $+t^{-1}$ }
\put(202,1){to the self intersection of $S$.}
\end{picture}
\caption{The lifts of the $+2$ push-off of $\gamma_2'$ in the cover of ${V}$ together with a  disk, $D$, bounded by a lift of $\gamma_2'$.}\label{fig:isect'n}
\end{figure}

By counting these intersection points, the twisted signature of $W'$ is equal to the $L^2$ signature of the $1\times 1$ matrix, $[t+t^{-1}]$.  A Fourier transform (see \cite[example 1.15]{L2algebraic}) send this matrix to the $1\times1$ matrix $[z+z^{-1}]$ over $L^2(S^1)$ ($S^1$ is the unit circle in $\mathbb C$ with normalized Lebesgue measure).  The signature of this matrix is equal to the measure of the subset of $S^1$ on which \begin{equation}z+z^{-1} = z+\overline z = 2\real(z)\end{equation} is positive minus the measure of the subset on which it is negative.  These sets have equal measure and so \begin{equation}\label{signature of W'}\sigma^2(W';\phi)=0\end{equation}

Combining equations (\ref{Lx < L2}), (\ref{L2 < L'}), (\ref{W computes L'}), (\ref{W equals W'}) and (\ref{signature of W'}) for each $x\ge2$ 
\begin{equation}
\begin{array}{rl}
\rho^0(L_x)&\le\rho^0(L_2)\le\rho^0(L')=\sigma^{(2)}(W)-\sigma(W)\\&=\sigma^{(2)}(W')-\sigma(W)=0-2<-1
\end{array}
\end{equation} and the proof is complete.

\end{proof}

\subsection{proof of Lemma~\ref{surgery}}

We now state and prove a marginally stronger version of Lemma~\ref{surgery}.  As promised it is not a difficult theorem to prove.
\begin{lemma}\label{surgery stronger}
\begin{enumerate}
\item
If a link $L'$ is given by performing $+1$ surgery on $L$ along a nullhomologous curve $\gamma$ then $\rho^0(L')\le\rho^0(L)\le \rho^0(L')+2$.
\item
If $L'$ is given by performing $-1$ surgery on $L$ along a nullhomologous curve then $\rho^0(L')-2\le\rho^0(L)\le \rho^0(L')$.
\end{enumerate}
\end{lemma}\begin{proof}
Let $W$ be the 4-manifold obtained by adding to $M(L)\times[0,1]$ a 2-handle along the $+1$ framing of $\gamma$ in $M(L)\times\{1\}$.  $\bdry(W)$ is given by $-M(L)\sqcup M(L')$ and both inclusions induce first homology isomorphisms.  The intersection form on $H_2(W)/H_2(\bdry W)$ is given by the $1\times 1$ matrix whose only entry is $1$ so $W$ has regular signature $\sigma(W)=1$.  

Thus, \begin{equation}\label{surgery equation}\rho^0(L')-\rho^0(L) = \sigma^{(2)}(W,\phi)-1,\end{equation} where $\phi:\pi_1(W)\to H_1(W)$ is the abelianization map.  $(W,\bdry_-W)$ has only one 2-handle so \begin{equation}|\sigma^{(2)}(W,\phi)|\le 1.\end{equation}  Rearranging~(\ref{surgery equation}) and applying this bound \begin{equation}\left|\rho^0(L')-\rho^0(L)+1\right|\le 1\end{equation} and $\rho^0(L')\le\rho^0(L)\le \rho^0(L')+2$, completing the proof of the first claim

The proof of the second claim is identical except that the signature of the bounded 4-manifold is $-1$.
\end{proof}

We close with some related questions which we would like to address.

\begin{question}
 What about the case $x=1$?  The $-3$ twist knot is of infinite order in the concordance group (proven by \cite[Corollary 1.2]{Tamulis} using Casson Gordon invariants).  If one follows the technique in this paper in the case $x=1$, then one finds that $\rho^0(L_1)=-1$, so $\rho^1(T_{-3})\in[-1,0]$.  Is there another choice of metabolizing link which has $\rho^0<-1$?  Is it the case that  $\rho^1(T_{-3})= 0$?
\end{question}

\begin{question}

 There are many more twist knots of algebraic order 2 whose concordance orders are unknown.  If one can find a derivative link for the connected sum of each such twist knot with itself and can compute the associated $\rho^0$-invariant, then one will have bounds on the $\rho^1$-invariant of the twist knots.  Presumably these bounds will imply that most of the twist knots have non-vanishing $\rho^1$-invariant.  If one can do this then one will have shown that most of the twist knots form a  linearly independent subset of the concordance group.
 \end{question}

%
%
%
\bibliographystyle{plain}

\bibliography{biblio}

\begin{thebibliography}{10}

\bibitem{CG1}
A.~J. Casson and C.~McA. Gordon.
\newblock On slice knots in dimension three.
\newblock In {\em Algebraic and geometric topology (Proc. Sympos. Pure Math.,
  Stanford Univ., Stanford, Calif., 1976), Part 2}, Proc. Sympos. Pure Math.,
  XXXII, pages 39--53. Amer. Math. Soc., Providence, R.I., 1978.

\bibitem{Cha3}
Jae~Choon Cha.
\newblock Topological minimal genus and {$L\sp 2$}-signatures.
\newblock {\em Algebr. Geom. Topol.}, 8(2):885--909, 2008.

\bibitem{C}
Tim~D. Cochran.
\newblock Noncommutative knot theory.
\newblock {\em Algebr. Geom. Topol.}, 4:347--398, 2004.

\bibitem{derivatives}
Tim~D. Cochran, Shelly Harvey, and Constance Leidy.
\newblock Derivatives of knots and second order signatures.
\newblock {\em Geometry and Topology}, 10(10):739--787, 2010.

\bibitem{whitneytowers}
Tim~D. Cochran, Kent~E. Orr, and Peter Teichner.
\newblock Knot concordance, {W}hitney towers and {$L\sp 2$}-signatures.
\newblock {\em Ann. of Math. (2)}, 157(2):433--519, 2003.

\bibitem{structureInConcordance}
Tim~D. Cochran, Kent~E. Orr, and Peter Teichner.
\newblock Structure in the classical knot concordance group.
\newblock {\em Comment. Math. Helv.}, 79(1):105--123, 2004.

\bibitem{CT}
Tim~D. Cochran and Peter Teichner.
\newblock Knot concordance and von {N}eumann {$\rho$}-invariants.
\newblock {\em Duke Math. J.}, 137(2):337--379, 2007.

\bibitem{Collins}
Julia Collins.
\newblock The ${L}^{(2)}$-signature of torus knots.
\newblock preprint avilable at http://front.math.ucdavis.edu/1001.1329.

\bibitem{DaKi}
James~F. Davis and Paul Kirk.
\newblock {\em Lecture notes in algebraic topology}, volume~35 of {\em Graduate
  Studies in Mathematics}.
\newblock American Mathematical Society, Providence, RI, 2001.

\bibitem{FrLM}
S.~{Friedl}, C.~{Leidy}, and L.~{Maxim}.
\newblock {{L}\^{}2-{B}etti numbers of plane algebraic curves}.
\newblock {\em Algebraic and Geometric Topology}, 4:893--934, 2004.

\bibitem{Go2}
C.~McA Gordon.
\newblock Ribbon concordance of knots in the three-sphere.
\newblock {\em Ann of Math.}, 257(2):157--170, 1981.

\bibitem{Ha2}
Shelly~L. Harvey.
\newblock Homology cobordism invariants and the {C}ochran-{O}rr-{T}eichner
  filtration of the link concordance group.
\newblock {\em Geom. Topol.}, 12(1):387--430, 2008.

\bibitem{Ji1}
Bo~Ju Jiang.
\newblock A simple proof that the concordance group of algebraically slice
  knots is infinitely generated.
\newblock {\em Proc. Amer. Math. Soc.}, 83(1):189--192, 1981.

\bibitem{Ka3}
Akio Kawauchi.
\newblock {\em A survey of knot theory}.
\newblock Birkh\"auser Verlag, Basel, 1996.
\newblock Translated and revised from the 1990 Japanese original by the author.

\bibitem{BlDuality}
C.~Kearton.
\newblock {B}lanchfield duality and simple knots.
\newblock {\em Transaction of the Americal Mathematical Society}, 202, 1975.

\bibitem{polynomialSplittingOfCG}
Se-Goo Kim.
\newblock Polynomial splittings of {C}asson-{G}ordon invariants.
\newblock {\em Math. Proc. Cambridge Philos. Soc.}, 138(1):59--78, 2005.

\bibitem{polynomialSplittingOfRho}
Se-Goo Kim and Taehee Kim.
\newblock Polynomial splittings of metabelian von {N}eumann rho-invariants of
  knots.
\newblock {\em Proc. Amer. Math. Soc.}, 136(11):4079--4087, 2008.

\bibitem{L5}
J.~Levine.
\newblock Knot cobordism groups in codimension two.
\newblock {\em Comment. Math. Helv.}, 44:229--244, 1969.

\bibitem{Le10}
J.~P. Levine.
\newblock Signature invariants of homology bordism with applications to links.
\newblock In {\em Knots 90 ({O}saka, 1990)}, pages 395--406. de Gruyter,
  Berlin, 1992.

\bibitem{Lis1}
Paolo Lisca.
\newblock Sums of lens spaces bounding rational balls.
\newblock {\em Algebr. Geom. Topol.}, 7:2141--2164, 2007.

\bibitem{LiN1}
Charles Livingston and Swatee Naik.
\newblock Obstructing four-torsion in the classical knot concordance group.
\newblock {\em J. Differential Geom.}, 51(1):1--12, 1999.

\bibitem{knotconcordanceandtorsion}
Charles Livingston and Swatee Naik.
\newblock Knot concordance and torsion.
\newblock {\em Asian Journal of Mathemematics}, 5:161--168, 2001.

\bibitem{L2invts}
Wolfgang L\"uck.
\newblock ${L}^2$ invariants of regular coverings of compact manifolds and
  cw-complexes.
\newblock In {\em Handbook of Geometric Topology}, pages 735--817.
  North-Holland, Amsterdam, 2002.

\bibitem{L2algebraic}
Wolfgang L\"uck.
\newblock ${L}^2$ invariants from the algebraic point of view.
\newblock In Kropholler P.~H. Bridson, M and I.~Leary, editors, {\em PJ., LMS
  Lecture Notes Series 358}, pages 271--277. Cambridge University Press, 2003.

\bibitem{L2sign}
Wolfgang L\"uck and Thomas Schick.
\newblock Various ${L}^2$ signatures and a topological ${L}^2$ signature
  theorem.
\newblock {\em High-dimensional manifold topology}, pages 362--399, 2003.

\bibitem{Tamulis}
A.~Tamulis.
\newblock {Knots of Ten or Fewer Crossings of Algebraic Order Two}.
\newblock {\em Journal of Knot Theory and Its Ramifications}, 11(2):211--222,
  2002.

\end{thebibliography}

%

\end{document}